\documentclass{amsart}
\usepackage[hmargin=30mm, vmargin=30mm]{geometry}
\usepackage[english]{babel}
\usepackage{amsmath,amssymb}
\usepackage{hyperref}
\usepackage{color}

\newtheorem{thmintro}{Theorem}

\newtheorem{theorem}{Theorem}[section]
\newtheorem{corollary}[theorem]{Corollary}
\newtheorem{lemma}[theorem]{Lemma}
\newtheorem{prop}[theorem]{Proposition}
\theoremstyle{definition}
\newtheorem{definition}[theorem]{Definition}
\newtheorem{remark}[theorem]{Remark}
\newtheorem{example}[theorem]{Example}
\newtheorem*{notation}{Notation}

\numberwithin{equation}{section}

\newcommand{\Rcal}{\mathcal{R}}
\newcommand{\RRR}{\mathrm{R}}
\newcommand{\LL}{\mathrm{L}}
\newcommand{\LLL}{\mathcal{L}}
\newcommand{\CC}{\mathbb{C}}
\newcommand{\NN}{\mathbb{N}}
\newcommand{\RR}{\mathbb{R}}
\newcommand{\g}{\mathfrak{g}}
\newcommand{\n}{\mathfrak{n}}
\newcommand{\hh}{\mathfrak{h}}
\newcommand{\inv}{^{-1}}
\newcommand{\co}{\colon\thinspace}
\newcommand{\norm}[1]{|\!|#1|\!|}
\newcommand{\Bignorm}[1]{\Big|\!\Big|#1\Big|\!\Big|}

\DeclareMathOperator{\Ad}{Ad}
\DeclareMathOperator{\Int}{Int}
\DeclareMathOperator{\ad}{ad}
\DeclareMathOperator{\Aut}{Aut}
\DeclareMathOperator{\Diff}{Diff}
\DeclareMathOperator{\GL}{GL}
\DeclareMathOperator{\Lie}{L}
\DeclareMathOperator{\Hom}{Hom}
\DeclareMathOperator{\End}{End}
\DeclareMathOperator{\Evol}{Evol}
\DeclareMathOperator{\evol}{evol}

\DeclareMathOperator{\id}{id}
\DeclareMathOperator{\dist}{d}
\DeclareMathOperator{\Mult}{Mult}

\DeclareMathOperator{\dd}{d\!}

\begin{document}

\renewcommand{\proofname}{{\bf Proof}}

\title{Half-Lie groups}
\author[Timoth\'ee Marquis]{Timoth\'ee Marquis$^*$}
\address{Department Mathematik, FAU Erlangen-Nuernberg, Cauerstrasse 11, 91058 Erlangen, Germany}
\email{marquis@math.fau.de}
\thanks{$^*$Supported by a Marie Curie Intra-European Fellowship}

\author[Karl-Hermann Neeb]{Karl-Hermann Neeb$^\dagger$}
\address{Department Mathematik, FAU Erlangen-Nuernberg, Cauerstrasse 11, 91058 Erlangen, Germany}
\email{neeb@math.fau.de}
\thanks{$^\dagger$Supported by DFG-grant NE 413/9-1, 
``Invariante Konvexit\"at in unendlich-dimensionalen Lie-Algebren''} 

\subjclass[2010]{22E65, 58B25} 
\keywords{Topological group, Lie group, regular Lie group}

\begin{abstract} In this paper we study the Lie theoretic properties of 
a class of topological groups 
which carry a Banach manifold structure but whose multiplication is not 
smooth. If $G$ and $N$ are Banach--Lie groups and 
$\pi : G \to \Aut(N)$ is a homomorphism defining a continuous action 
of $G$ on $N$, then $H := N \rtimes_\pi G$ is a Banach manifold with 
a topological group structure for which the left multiplication maps  
are smooth, but the right multiplication maps need not to be. We show that 
these groups share surprisingly many properties with Banach--Lie groups: 
(a) for every regulated function $\xi : [0,1] \to \hh$ the initial value problem 
$\dot\gamma(t) = \gamma(t)\xi(t)$, $\gamma(0)= 1_H$, has a solution 
and the corresponding evolution map from curves in $\hh$ to curves in 
$H$ is continuous; (b) every $C^1$-curve $\gamma$ with $\gamma(0) = 1$ and 
$\gamma'(0) = x$ satisfies $\lim_{n \to \infty} \gamma(t/n)^n = \exp(tx)$; 
(c) the Trotter formula holds for $C^1$ one-parameter groups in $H$; 
(d) the subgroup $N^\infty$ of elements with smooth $G$-orbit maps in $N$ 
carries a natural Fr\'echet--Lie group structure for which the 
$G$-action is smooth; (e) the resulting Fr\'echet--Lie group 
$H^\infty := N^\infty \rtimes G$ is also regular in the sense of (a). 
\end{abstract}

\maketitle

\section{Introduction}

The theory of infinite dimensional Lie groups can be 
developed very naturally in the context of Lie groups modelled on 
locally convex spaces, so called {\it locally convex Lie groups}. 
For more details on this theory, we recommend the survey 
article \cite{NeJap} or the forthcoming monograph \cite{GlNe16}. 
The theory of locally convex Lie groups has, however, certain drawbacks, 
the most serious one being that the Inverse and Implicit Function Theorem fail 
beyond the class of Banach manifolds. In some situations one can still 
use the Nash--Moser Theorem, but this theorem is difficult to apply because 
its assumptions are often hard to verify. 

It is for this reason that, early on in infinite dimensional Lie
theory, people have tried to ``approximate'' Fr\'echet--Lie groups 
by certain Banach manifolds to work in a context where the 
analytic tools, such as existence of solutions of ODEs and inverse function 
results, can be applied, and then perform a passage to the Fr\'echet limit, 
which often is a projective limit of topological groups. 
The most prominent situation where this strategy has been applied 
with great success is the analysis of diffeomorphism groups 
of compact smooth manifolds $M$. The group $\Diff(M)$ of smooth 
diffeomorphisms carries the structure of a Fr\'echet--Lie group, 
but the usual construction of charts also applies to the groups 
$\Diff^k(M)$ of $C^k$-diffeomorphisms for any $k  \geq 1$. 
This provides the structure of a $C^k$-manifold structure on each 
of the topological groups $\Diff^k(M)$, but 
neither multiplication nor inversion are smooth. Only the 
right multiplications are smooth maps. This kind of ``weak Lie group structure'' 
is usually dealt with in the context of 
{\it ILB (inverse limit of Banach) Lie groups}, 
which play an important role in geometric analysis (cf.\ \cite{AK98}, \cite{EMi99}). 
The Lie theory of these groups has been developed 
by H.~Omori and his collaborators in a series of papers 
culminating in \cite{MOKY85} (see also Omori's monograph \cite{Omori97}). 

Another context where Lie theory leaves its natural ``smooth context'' is 
in the theory of group representations. There one studies 
representations $\pi : G \to \GL(V)$ of a Lie group $G$ on a 
Fr\'echet space $V$, for which the action of $G$ on $V$ is continuous, but 
not in general smooth. This seemingly weak requirement is dictated by 
the applications where smoothness of the action of $G$ on $V$ would 
be much too strong. This phenomenon is well-known from  the theory 
of one-parameter semigroups on Banach spaces, where norm continuity 
is much too strong and strong continuity is the natural regularity 
assumption. Applying Lie theoretic methods to continuous 
representations can be a difficult task, but recently some quite effective 
tools to overcome these difficulties have been developed 
(see in particular \cite{NeSa13}). To apply these tools, one has to assume that the addition 
in the Lie algebra $\Lie(G)$ of the Lie group $G$ under consideration 
is compatible with the topological group structure in the sense that 
$G$ has the {\it Trotter property}, i.e., for every $x_1, x_2 \in \Lie(G)$, 
\[ \exp_G(t(x_1 + x_2)) = \lim_{n \to \infty} 
\Big(\exp_G\Big(\frac{t}{n}x_1\Big)\exp_G\Big(\frac{t}{n}x_2\Big)\Big)^n \] 
holds uniformly on compact subsets of $\RR$. 
All locally exponential Lie groups (i.e., groups for which the exponential function 
is a local diffeomorphism in $0$) have this property, and this includes in 
particular all Banach--Lie groups (\cite{NeJap}). Beyond the Banach context, 
the Trotter property is often hard to verify, but in \cite{NeSa13} this is 
done for diffeomorphism groups of compact manifolds and the Virasoro group. 
These examples already show that the Trotter property is much weaker 
than the local exponentiality of the group, and even more so, than 
the local convergence of the Baker--Campbell--Dynkin--Hausdorff 
series in a $0$-neighbourhood of the Lie algebra.
Much of this has recently been facilitated by H.~Gl\"ockner's new regularity 
results \cite{Gl15} which provide also solutions to differential 
equations of the form $\dot\gamma(t) = \gamma(t)\xi(t)$, where 
$\xi$ is not necessarily continuous. So $\xi$ could also be a Riemannian 
step function, or a uniform limit of step functions, i.e., a regulated 
function. Recently M.~Hanusch succeeded in showing that the strong Trotter 
property, i.e., that every $C^1$-curve $\gamma$ with $\gamma(0) = 1$ and 
$\gamma'(0) = x$ satisfies $\lim_{n \to \infty} \gamma(t/n)^n = \exp(tx)$,  
follows from the local $\mu$-convexity of $G$ (\cite{Ha18}). 
This  is a continuity 
requirement on the multiplication expressed in terms of seminorms 
and local charts which is intimately related to regularity properties 
of the Lie group (\cite{Ha17}). 
 \\ 

In the present paper we pursue a more detailed analysis of a 
class of Banach manifolds which carry a topological group 
structure but which are not Lie groups, namely semidirect products 
$H := N \rtimes_\pi G$, where both $N$ and $G$ are Banach--Lie groups,
 but the homomorphism $\pi : G \to \Aut(N)$ only defines a continuous 
action $\pi^{\wedge}\co G\times N\to N, (g,n)\mapsto\pi(g)n$. 
Then $H$ is a topological group and a smooth Banach manifold. 
The multiplication on $H$, however, is in general \emph{not} smooth as the right multiplication maps $$\rho_{n_2,g_2}\co H\to H, \ (n_1,g_1)\mapsto (n_1,g_1)(n_2,g_2)=(n_1\cdot \pi(g_1)n_2,g_1g_2)$$ are, in general, only continuous.  
On the other hand, the left multiplication maps are smooth. 
Following the terminology of \cite{KMR15}, we call such a topological group, with a Banach manifold structure and smooth left multiplication maps, a \emph{(left) half-Lie group}.
The semidirect products $H = N \rtimes_\pi G$ as above constitute an 
important class of examples of half-Lie groups because they are still rather 
well-behaved but they also display many of the pathologies of half-Lie 
groups that are not Lie groups. Here 
already the case $G =\RR$ is very interesting. 
Another interesting class of examples arises for 
the group $N = C^k(M,K)$, $M$ a compact smooth manifold, 
$K$ a Banach--Lie group and $k \in \NN_0$, where the action of $G$ on $N$ 
comes from a smooth action of $G$ on $M$. 
The aim of this paper is to understand to which extent the 
Lie theoretic properties of Banach--Lie groups survive 
in the framework of these half-Lie groups.

We now describe our main results in more detail.
The first problem to investigate is the existence of an \emph{exponential function} $\exp_H\co \hh\to H$ on the tangent space $\hh:=T_1H=T_1N\times T_1G$ of $H$ at the identity, that is, of a map $\exp_H\co\hh\to H$ such that for each $x\in\hh$, the curve $\gamma_x\co I=[0,1]\to H$ defined by $\gamma_x(t):=\exp_H(tx)$ is a $C^1$ solution to the initial value problem (IVP)
$$\gamma_x'(t)=\gamma_x(t).x, \quad \gamma_x(0)=1_H,$$
where we denoted for each $h\in H$ by $\hh\to T_h(H),  
x\mapsto h.x:=T_1(\lambda_h)x$ the action of the tangent map of 
$\lambda_h\co H\to H, g\mapsto hg$ at the identity on $\hh$. More generally, for each continuous curve $\gamma\in C^0(I,\hh)$, one may ask whether there exists a solution $\eta\in C^1(I,H)$ to the IVP
\begin{equation}\label{eqn:IVP_intro}
\eta'(t)=\eta(t).\gamma(t), \quad \eta(0)=1_H.
\end{equation}
If a solution to (\ref{eqn:IVP_intro}) exists, then it is unique (see \cite[\S II.3]{NeJap}), yielding an \emph{evolution map} 
$$\Evol_H\co C^0(I,\hh)\to C^1(I,H), \ \gamma\mapsto\eta.$$ If, moreover, $\Evol_H$ is continuous, then the group $H$ is called \emph{$C^0$-regular}. For instance, Banach--Lie groups are $C^0$-regular (see \cite[Theorem~C]{Gl15}).

In \cite{Gl15}, H.~Gl\"ockner further defined a concept of \emph{$R$-regularity} (which implies $C^0$-regularity), by replacing the space $C^0(I,\hh)$ in the above definition with the space $R(I,\hh)$ of \emph{regulated functions}, that is, of functions in $\LL^{\infty}(I,\hh)$ that are uniform limits of step functions (see \S\ref{section:RReg} below for more details 
and the precise meaning of the IVP (\ref{eqn:IVP_intro}) in this context). H.~Gl\"ockner then proves that Banach--Lie groups are $R$-regular, and derives various 
important consequences. 

Our first result implies in particular that the half-Lie group $H:=N\rtimes_{\pi}G$ possesses a (continuous) exponential map $\exp_H\co \hh\to H$ (see \S\ref{section:RROHH}):

\begin{thmintro}\label{thmintro:Rreg}
Let $G,N$ be Banach--Lie groups, and let $\pi\co G\to \Aut(N)$ define a continuous action. Then the half-Lie group $H:=N\rtimes_{\pi}G$ is $R$-regular.
\end{thmintro}
\noindent

\medskip

A second natural problem is to understand whether $\hh$ admits a Lie algebra structure, and whether such a structure can be, as in the classical case, reconstructed from the space $\Hom^1(\RR,H)$ of $C^1$ one-parameter subgroups of $H$. We recall that for any Banach--Lie group $\Gamma$ with Lie algebra $\Lie(\Gamma)$ and exponential function $\exp_{\Gamma}\co \Lie(\Gamma)\to\Gamma$, the Lie algebra structure on $\Lie(\Gamma)$ can be obtained by using the identification $\Hom^1(\RR,\Gamma)\stackrel{\approx}{\longrightarrow}\Lie(\Gamma), \gamma\mapsto \gamma'(0)$, and the fact that $\Gamma$ has the \emph{Trotter property}, i.e., 
for all $\gamma_1,\gamma_2\in\Hom^1(\RR,\Gamma)$,
\begin{equation*}
\lim_{n\to\infty}\big(\gamma_1(\tfrac{t}{n})\gamma_2(\tfrac{t}{n})\big)^n=\exp_{\Gamma}\big(t(\gamma_1'(0)+\gamma_2'(0))\big)
\end{equation*}
holds uniformly in $t$ on compact subsets of $\RR$, 
as well as the \emph{commutator property}, i.e., for all $\gamma_1,\gamma_2\in\Hom^1(\RR,\Gamma)$,
\begin{equation*}
\lim_{n\to\infty}\big(\gamma_1(\tfrac{\sqrt{t}}{n})\gamma_2(\tfrac{\sqrt{t}}{n})\gamma_1(-\tfrac{\sqrt{t}}{n})\gamma_2(-\tfrac{\sqrt{t}}{n})\big)^{n^2}=\exp_{\Gamma}\big(t[\gamma_1'(0),\gamma_2'(0)]\big)
\end{equation*}
holds uniformly in $t$ on compact subsets of $[0,\infty[$. Actually $\Gamma$ has 
the \emph{strong Trotter property} (which implies both the Trotter and commutator properties within the class of locally convex Lie groups, see \cite[Theorem~H]{Gl15}), that is, for each $C^1$-curve $\gamma\co I\to \Gamma$ with $\gamma(0)=1_{\Gamma}$, 
\begin{equation*}
\lim_{n\to\infty}\gamma(\tfrac{t}{n})^n=\exp_{\Gamma}(t\gamma'(0))
\end{equation*}
uniformly in $t$ on compact subsets of $[0,\infty[$.

It turns out that the space $\hh$ carries, in general, no natural Lie bracket, 
so that one cannot speak of 
the ``Lie algebra of $H$'' (see \S\ref{section:TFFH}). On the other hand, we show, as in \cite[Theorem~I]{Gl15}, that the $R$-regularity of $H$ implies that $H$ has the strong Trotter property.

\begin{thmintro}\label{thmintro:STP}
Let $G,N$ be Banach--Lie groups, and let $\pi\co G\to \Aut(N)$ define a continuous action. Then the half-Lie group $H:=N\rtimes_{\pi}G$ has the strong Trotter property.
\end{thmintro}

In our setting, the strong Trotter property does not immediately imply the Trotter property, as for two $C^1$-curves $\gamma_1,\gamma_2\co\RR\to H$, the curve $\gamma(t):=\gamma_1(t)\gamma_2(t)$ need not be $C^1$. Nevertheless, with some extra work, we can show that $H$ has the Trotter property as well. 

\begin{thmintro}\label{thmintro:Trotter}
Let $G,N$ be Banach--Lie groups, and let $\pi\co G\to \Aut(N)$ define a continuous action. Then the half-Lie group $H:=N\rtimes_{\pi}G$ has the Trotter property.
\end{thmintro}
\noindent
We actually prove a stronger result, generalising both Theorems~\ref{thmintro:STP} and \ref{thmintro:Trotter} (see Theorem~\ref{thm:TFimproved} below for a precise statement).
As a surprising side result, we further show in \S\ref{section:COPSOH} that, if 
$N$ is abelian, then any continuous one-parameter subgroup of $H$ is conjugate 
to a smooth one-parameter subgroup, hence of the form 
$t \mapsto g \exp_H(tx)g^{-1}$ for some $g \in H$, $x \in \hh$.

\medskip

In representation theory, an important technique consists in the passage 
from a continuous representation on a Banach space to the subspace of 
smooth vectors, i.e., the elements with smooth orbit maps, on which the Lie algebra 
acts naturally. In this context, a third problem is to ask whether the subgroup 
$$N^{\infty}:=\{n\in N \ | \ \textrm{$G\to N, \ g\mapsto \pi(g)n$ is smooth}\}$$ of \emph{smooth elements} of $N$ for the action $\pi$ carries a natural Lie group structure. Building on results from \cite{Ne10b}, where the above question is shown to have a positive answer
when $N$ is a Banach space, we prove that $N^{\infty}$ has a Fr\'echet--Lie group structure for which the induced action ${\pi_{\infty}\co G\to\Aut(N^{\infty})}, g\mapsto\pi(g)|_{N^{\infty}}$ is smooth. This implies in particular the following (see \S\ref{section:FLGSOH}).

\begin{thmintro}
Let $G,N$ be Banach--Lie groups and $\pi\co G\to \Aut(N)$ define a continuous action. 
Then $N^{\infty}$ carries a natural Fr\'echet--Lie group structure for which the 
action $\pi_\infty$ is smooth. In particular, 
the semidirect product 
$H^{\infty}:=N^{\infty}\rtimes_{\pi_{\infty}}G$ is a Fr\'echet--Lie group. 
\end{thmintro}

Finally, we investigate in \S\ref{section:RROH} the $R$-regularity of the Lie group $H^{\infty}$.

\begin{thmintro}
Let $G,N$ be Banach--Lie groups, and let $\pi\co G\to \Aut(N)$ define a continuous action. Then $H^{\infty}:=N^{\infty}\rtimes_{\pi_{\infty}}G$ is $R$-regular with a smooth evolution map. In particular, $H^{\infty}$ has the strong Trotter and commutator properties.
\end{thmintro}

\subsubsection*{Acknowledgement} 
The authors thank Helge Gl\"ockner for enlightening 
 discussions on the topic of measurable regularity properties. 
They also thank the referees for helpful remarks and for pointing out 
some references.

\section{Preliminaries}
\begin{notation}
Throughout this paper, $\NN=\{1,2,\dots\}$ denotes the set of positive integers, and $\NN_0:=\NN\cup\{0\}$ the set of nonnegative integers.
\end{notation}

We first recall the basic concepts pertaining to infinite-dimensional Lie groups modelled on locally convex spaces, and their measurable regularity properties.
The main references for this section are \cite{Gl15} and \cite{NeJap}.

\subsection{Lebesgue spaces \texorpdfstring{(\cite[1.7--1.13, 1.25, 1.31]{Gl15})}{}}
Let $I=[a,b]\subseteq\RR$ for some $a<b$, which we view as a measure space for the (restriction of) the Lebesgue measure on $\RR$. Let $E$ be a real Fr\'echet space, which we view as a measurable space with respect to its $\sigma$-algebra of Borel sets. We write $P(E)$ for the set of all continuous seminorms $q\co E\to [0,\infty[$. 

 We let $\LLL^1(I,E)$ denote the space of all measurable functions $\gamma\co I\to E$ with separable image (i.e. $\gamma(I)$ has a dense countable subset) such that
$$\norm{\gamma}_{\LLL^1,q}:=\int_{I}{q(\gamma(s))ds}<\infty \quad\textrm{for all $q\in P(E)$.}$$ Similarly, we let $\LLL^{\infty}(I,E)$ denote the space of all measurable functions $\gamma\co I\to E$ with separable and bounded image, so that
$$\norm{\gamma}_{\LLL^{\infty},q}:=\norm{q\circ\gamma}_{\LLL^{\infty}}=\mathrm{ess} \ \mathrm{sup}(q\circ\gamma)<\infty \quad\textrm{for all $q\in P(E)$.}$$ For $p\in\{1,\infty\}$, we equip $\LLL^p(I,E)$ with the (non-Hausdorff) locally convex vector topology defined by the seminorms $\norm{\cdot}_{\LLL^p,q}$ for $q\in P(E)$.

Let $\LLL^{\infty}_{rc}(I,E)$ be the space of all measurable maps $\gamma\co I\to E$ with relatively compact image. We endow $\LLL^{\infty}_{rc}(I,E)\subseteq \LLL^{\infty}(I,E)$ with the topology induced by $\LLL^{\infty}(I,E)$. Finally, let $\mathcal R(I,E)$ be the space of functions $\gamma\co I\to E$ that are the uniform limit of a sequence of step functions. We recall that $\gamma\co I\to E$ is a \emph{step function} if there exists a partition $a=t_0<t_1<\dots <t_n=b$ of $I$ such that $\gamma|_{]t_{j-1},t_j[}$ is constant for all $j\in\{1,\dots,n\}$. A function $\gamma\in \mathcal R(I,E)$ is called \emph{regulated}, and we endow $\mathcal R(I,E)\subseteq \LLL^{\infty}_{rc}(I,E)$ with the topology induced by $\LLL^{\infty}(I,E)$.

Given a measurable map $\gamma\co I\to E$, we write $[\gamma]$ for the equivalence class of measurable maps $\gamma_1\co I\to E$ such that $\gamma(s)=\gamma_1(s)$ for almost all $s\in I$. (When no confusion is possible, we will also simply write $\gamma$ for its equivalence class $[\gamma]$.) We then define $\LL^{1}(I,E)$ (resp. $\LL^{\infty}(I,E)$, $\LL^{\infty}_{rc}(I,E)$, $\RRR(I,E)$) as the space of equivalence classes $[\gamma]$ with $\gamma$ in $\LLL^{1}(I,E)$ (resp. $\LLL^{\infty}(I,E)$, $\LLL^{\infty}_{rc}(I,E)$, $\mathcal R(I,E)$). For $p\in \{1,\infty\}$, we equip $\LL^{p}(I,E)$ with the locally convex vector topology defined by the seminorms $\norm{[\gamma]}_{\LL^p,q}:=\norm{\gamma}_{\LLL^p,q}$ for $q\in P(E)$, and we give $\LL^{\infty}_{rc}(I,E)$ and $R(I,E)$ the induced topology, coming from the inclusions $\RRR(I,E)\subseteq \LL^{\infty}_{rc}(I,E)\subseteq \LL^{\infty}(I,E)$. Note that $R(I,E)$ admits a basis of open $0$-neighbourhoods consisting of the sets $$R(I,V):=\{\alpha\in R(I,E) \ | \ \alpha(I)\subseteq V\},$$ where $V$ runs through a basis of open $0$-neighbourhoods in $E$.

Finally, note that the map $C(I,E)\to \LL^{\infty}(I,E), \gamma\mapsto [\gamma]$ is injective; we will equip the space $C(I,E)$ of continuous functions with the induced topology, given in this case by the seminorms $\norm{\gamma}_{\LLL^{\infty},q}=\sup_{t\in I}{q(\gamma(t))}$ for $q\in P(E)$. 

\subsection{Integration \texorpdfstring{(\cite[1.16--1.28]{Gl15})}{}}\label{section:integration}
Let $E$ be a real locally convex space, $E'$ be its topological dual (that is, the space of all continuous linear functionals $E\to\RR$), and let $\gamma\co I=[a,b]\to E$ be a function such that $\lambda\circ\gamma\in \LLL^{1}(I,\RR)$ for each $\lambda\in E'$. We define the \emph{weak integral} of $\gamma$, if it exists, as the unique element $w\in E$ such that 
$$\lambda(w)=\int_{I}{\lambda(\gamma(t))dt}\quad\textrm{for all $\lambda\in E'$},$$
and we write $\int_a^b\gamma(t)dt=\int_I\gamma(t)dt:=w$. 

As usual, a map $\eta\co I\to E$ is called \emph{differentiable at $t\in I$} if the limit $\eta'(t):=\lim_{s\to t}{\frac{\eta(s)-\eta(t)}{s-t}}$ exists in $E$. We have the following version of the Fundamental Theorem of Calculus: 

\begin{lemma}\label{lemma:thm_fond_CDI}
Let $E$ be a Fr\'echet space, $I=[a,b]\subseteq\RR$ and $\gamma\in\LLL^1(I,E)$. Then the weak integrals needed to define $$\eta\co I\to E, \ s\mapsto \int_{a}^s{\gamma(t)dt}$$ exist, and $\eta$ is a continuous function which is differentiable almost everywhere, with $\eta'=[\gamma]$.
\end{lemma}

\subsection{Differentiation \texorpdfstring{(\cite[1.49--1.52]{Gl15})}{}}\label{section:Diff}
Let $E$ and $F$ be real locally convex spaces, $U\subseteq E$ be an open set and $f\co U\to F$ be a map. 
The \emph{derivative of $f$ at $x\in U$ in the direction $y\in E$} is defined as the limit
$$df(x,y):=(D_yf)(x):=\frac{d}{dt}\Big|_{t=0}f(x+ty)=\lim_{t\to 0}\frac{1}{t}(f(x+ty)-f(x)),$$
whenever it exists. We say that $f$ is $C^0$ if it is continuous. We say that $f$ is $C^1$ if $f$ is continuous, the derivatives $df(x,y)$ exist in $F$ for all $(x,y)\in U\times E$, and $df\co U\times E\to F$ is continuous. Recursively, we say, for some integer $k\geq 1$, that $f$ is $C^k$ if $f$ is $C^1$ and $df\co U\times E\to F$ is $C^{k-1}$. Equivalently, $f$ is $C^k$ if and only if it is continuous and, for all positive integers $j\leq k$, the iterated directional derivatives 
$$d^jf(x,y_1,\dots,y_j):=(D_{y_j}\dots D_{y_1}f)(x)$$
exist for all $x\in U$ and $y_1,\dots,y_j\in E$, and the map $d^jf\co U\times E^j\to F$ is continuous. We call $f$ \emph{smooth} or $C^{\infty}$ if it is $C^{k}$ for all $k\in\NN$.

We record for future reference the following results.

\begin{lemma}[{\cite[2.1]{Gl15}}]\label{lemma:RtoR}
Let $E_1,E_2,F$ be Fr\'echet spaces, and let $f\co E_1\times E_2\to F$ be a continuous map such that $f(x,\cdot)\co E_2\to F$ is linear for all $x\in E_1$. Let $\eta\co I=[0,1]\to E_1$ be a continuous function and $\gamma\in\mathcal R(I,E_2)$. Then $f\circ(\eta,\gamma)\in\mathcal R(I,F)$.
\end{lemma}

\begin{lemma}\label{lemma:technical_lemma2} 
Let $E_1,E_2,F$ be Fr\'echet spaces and let $V\subseteq E_1$ be open. Let $f\co V\times E_2\to F$ be a smooth map such that $f(v,\cdot)\co E_2\to F$ is linear for all $v\in V$. Then 
$$\widetilde{f}\co V\times R(I,E_2)\to R(I,F), \ (x,\gamma)\mapsto f(x,\gamma(\cdot))$$ is smooth.
\end{lemma}
\begin{proof}
This follows from \cite[2.2]{Gl15}, since the natural injection $V\hookrightarrow C(I,V)$ is smooth.
\end{proof}

\subsection{Manifolds \texorpdfstring{(\cite[1.53]{Gl15}, \cite[Chapter I]{NeJap})}{}}
Since composition of smooth maps are smooth, one can define a \emph{smooth manifold $M$ modelled on a real locally convex space $E$} (or just \emph{$E$-manifold}) by replacing the modelling space $\RR^n$ by $E$ in the classical definitions of manifolds (see \cite{NeJap}). If $E$ is a Banach (resp. Fr\'echet) space, then $M$ is called a \emph{Banach} (resp. \emph{Fr\'echet}) \emph{manifold}. 

As usual, $TM=\dot{{\bigcup}}_{x\in M}{T_xM}$ then denotes the tangent bundle of $M$ and $T_xM\cong E$ the tangent space of $M$ at $x\in M$. Likewise, for a smooth map $f\co M\to  N$ between smooth manifolds, $Tf\co TM\to TN$ (resp. $T_xf\co T_xM\to T_{f(x)}N$) is the corresponding tangent map (resp. tangent map at $x\in M$). If $U$ is an open subset of a locally convex space $E$, we identify $TU$ with $U\times E$. For a smooth map $f\co M\to E$ from a smooth manifold $M$ to a locally convex space $E$, we then write $df\co TM\to E$ for the second component of $Tf\co TM\to E\times E$. Note that \S\ref{section:Diff} also yields a notion of \emph{$C^k$-maps} between smooth manifolds for each $k\in\NN\cup\{\infty\}$.

\subsection{Lie groups \texorpdfstring{(\cite[Chapters II--IV]{NeJap})}{}}\label{section:LG_NeJap}
A \emph{(locally convex) Lie group} $G$ is a group with a smooth manifold structure modelled on a locally convex space, for which the group operations (multiplication and inversion) are smooth. We write $1_G$ (or simply $1$ if no confusion is possible) for the identity element of $G$, and $\lambda_g\co G\to G, x\mapsto gx$ and $\rho_g\co G\to G,x\mapsto xg$ for the left and right multiplication maps. For $g\in G$, we also write $\Int(g)\co G\to G,x\mapsto gxg\inv$ for the conjugation map.

For each $x\in T_{1}G$, there is a unique left invariant vector field $x_l\co G\to TG$ with $x_l(1)=x$, defined by $x_l(g):=T_1(\lambda_g)x$. The Lie bracket on the space of left invariant vector fields then induces a continuous Lie bracket on $\g:=T_1(G)$, characterised by $[x,y]_l=[x_l,y_l]$ for $x,y\in\g$. We let $\Lie$ denote the functor from the category of locally convex Lie groups to the category of locally convex topological Lie algebras, which associates to a group $G$ its Lie algebra $\Lie(G):=(\g,[\cdot,\cdot])$ and to a Lie group morphism $\varphi\co G_1\to G_2$ the corresponding tangent map at the identity $\Lie(\varphi)=T_1(\varphi)\co\Lie(G_1)\to\Lie(G_2)$. If $\g$ is a Banach (resp. Fr\'echet) space, then $G$ is called a \emph{Banach} (resp. \emph{Fr\'echet}) \emph{Lie group}.

The left multiplication and conjugation maps on $G$ induce smooth maps
$$G\times\g\to TG, \ (g,x)\mapsto g.x:=T_1(\lambda_g)x\quad\textrm{and}\quad \Ad\co G\times\g\to\g, \ (g,x)\mapsto \Ad(g)x:=\Lie(\Int(g))x.$$
 Note that the adjoint action $\Ad$ of $G$ on $\g$ is by topological isomorphisms.

A map $\exp_G\co \g\to G$ is called an \emph{exponential function} if for any $x\in\g$, the curve $\gamma_x(t):=\exp_G(tx)$ is a $C^1$ one-parameter subgroup with $\gamma_x'(0):=T_0(\gamma_x)(1)=x$. If $G$ has an exponential function, then it is unique. 

Assume now that $G$ is a Banach Lie group. Then $G$ has a smooth exponential function $\exp_G\co \g\to G$, and $\exp_G$ maps some open (convex) $0$-neighbourhood $V_G$ in $\g$ diffeomorphically onto some open subset $U_G$ of $G$. If $W_G\subseteq V_G$ is an open (convex) $0$-neighbourhood in $\g$ such that $$\{\exp_G(x)\exp_G(y) \ | \ x,y\in W_G\}\subseteq V_G,$$ then one can define on $W_G$ the \emph{local multiplication}
$$*\co W_G\times W_G\to \g, \ (x,y)\mapsto x*y:=\exp_G\inv(\exp_G(x)\exp_G(y)),$$
which is a smooth map satisfying $tx*sx=(t+s)x$ for all $x\in W_G$ and $s,t\in\RR$ with $|s|,|t|,|s+t|\leq 1$ (\cite[Example~IV.2.4]{NeJap}).

Given some $k\in\NN\cup\{\infty\}$ with $k\geq 1$, some locally convex spaces $E,F$, some open subset $U\subseteq E$ and some $C^k$-map $f\co U_G\times U\to F$, we will use the exponential chart $(U_G,\exp_G\inv)$ of $G$ to define the \emph{$k$-th derivative $d_1^kf\co U_G\times\g^k\times U\to F$ of $f$ in the first coordinate} by setting
$$d_1^1f(g,x,v):=d_1f(g,x,v):=d(f^v)(g,x)=\frac{d}{dt}\Big|_{t=0}f(g\exp_G(tx),v)\quad\textrm{for all $g\in U_G$, $x\in\g$ and $v\in U$}$$
where $f^v\co U_G\to F,g\mapsto f(g,v)$ and, recursively, $d_1^{s+1}f=d_1(d_1^{s}f)$ for all $s\in\{1,\dots,k-1\}$. In other words, for all $g\in U_G$, $x_1,\dots,x_k\in\g$ and $v\in U$,
\begin{equation}\label{eqn:k-th_diff}
d_1^kf(g,x_1,\dots,x_k,v)=\frac{d}{dt_1}\Big|_{t_1=0}\dots\frac{d}{dt_k}\Big|_{t_k=0}f(g\exp_G(t_1x_1)\dots \exp_G(t_kx_k),v).
\end{equation}
Note that $d_1^sf$ is a $C^{k-s}$-map for all $s\in\{1,\dots,k\}$. For a $C^k$-map $h\co U_G\to F$, we also define $d^kh:=d_1^kh\co U_G\times\g^k\to F$ as above by viewing $h$ as a map $h\co U_G\times\{0\}\to F$.

\subsection{Half-Lie groups}\label{section:subSTG}
In this paper, we will consider the following generalisation of a Lie group (see also Section~\ref{section:STG} below). Let $E$ be a locally convex space, and let $G$ be a smooth manifold modelled on $E$. We call $G$ a \emph{(left) half-Lie group modelled on $E$} if $G$ admits a topological group structure (with respect to the manifold topology) such that all left multiplication maps $\lambda_g\co G\to G$ are smooth. 

\subsection{Local Lie groups \texorpdfstring{(\cite[II.1.10]{NeJap})}{}}
Given a group $G$ with multiplication $m\co G\times G\to G$ and identity $1_G$, a quadruple $(U,D_U,m,1_G)$ consisting of a symmetric subset $U=U\inv\subseteq G$ with $1_G\in U$ and of a subset $D_U\subseteq U\times U$ with $((U\times\{1_G\})\cup (\{1_G\}\times U)\subseteq D_U$ and $m(D_U)\subseteq U$ is a so-called \emph{local group}. 
If $(U,D_U,m,1_G)$ is such 
a local group and, in addition, $U$ has a smooth manifold structure, $D_U$ is open, and the local multiplication and inversion maps $D_U\to U,(x,y)\mapsto m(x,y)$ and $U\to U,x\mapsto x\inv$ are smooth, then $(U,D_U,m,1_G)$ (or simply $U$) is called a \emph{local Lie group}. 

\subsection{Absolutely continuous maps \texorpdfstring{(\cite[3.6--3.20, 4.2]{Gl15})}{}}\label{section:ACM}
Let $I=[a,b]\subseteq\RR$ for some $a<b$ and let $E$ be a Fr\'echet space. Define $AC_R(I,E)\subseteq C(I,E)$ as the space of continuous functions $\eta\co I\to E$ for which there exists $[\gamma]\in R(I,E)$ such that $$\eta(s)=\eta(a)+\int_{a}^s{\gamma(t)dt}\quad \textrm{for all $s\in I$.}$$ Then $[\gamma]=\eta'$ is unique by Lemma~\ref{lemma:thm_fond_CDI}, and the map $$AC_R(I,E)\to E\times R(I,E), \ \eta\mapsto (\eta(a),\eta')$$ is an isomorphism, which we use to define on $AC_R(I,E)$ a locally convex vector topology. The inclusion map $AC_R(I,E)\hookrightarrow C(I,E)$ is then continuous, and for any open subset $V\subseteq E$, the set 
$$AC_R(I,V):=\{\eta\in AC_R(I,E) \ | \ \eta(I)\subseteq V\}$$ is open in $AC_R(I,E)$. 

More generally, given a smooth manifold $M$ modelled on $E$, one can define $AC_R(I,M)$ as the set of all continuous functions $\eta\co I\to M$ for which there is a partition $a=t_0<t_1<\dots <t_n=b$ of $I$ such that, for each $j\in\{1,\dots,n\}$, there exists a chart $\varphi_j\co U_j\to V_j\subseteq E$ of $M$ with $\eta([t_{j-1},t_j])\subseteq U_j$ and $\varphi_j\circ\gamma|_{[t_{j-1},t_j]}\in AC_R([t_{j-1},t_j],E)$.

If $G$ is a Fr\'echet--Lie group, then $AC_R(I,G)$ is a group under pointwise multiplication, and there is a unique Lie group structure on $AC_R(I,G)$ such that
$$AC_R(I,U):=\{\eta\in AC_R(I,G) \ | \ \eta(I)\subseteq U\}$$
is open in $AC_R(I,G)$ and 
$$AC_R(I,\varphi)\co AC_R(I,U)\to AC_R(I,V), \ \eta\mapsto\varphi\circ\eta$$
is a diffeomorphism for each chart $\varphi\co U\to V$ of $G$ such that $1_G\in U$ and $U=U\inv$.

\subsection{Logarithmic derivative \texorpdfstring{(\cite[5.1--5.11]{Gl15})}{}}\label{section:LD}
Let $E$ be a Fr\'echet space. Let $G$ be either a half-Lie group or a local Lie group modelled on $E$. Write $\g:=T_{1}(G)$. Let $I=[a,b]\subseteq\RR$ for some $a<b$ and let $\eta\in AC_R(I,G)$. We define the \emph{(left) logarithmic derivative} $\delta(\eta)\in R(I,\g)$ of $\eta$ as follows.

Let $a=t_0<t_1<\dots <t_n=b$ be a partition of $I$ such that, for each $j\in\{1,\dots,n\}$, there exists a chart $\varphi_j\co U_j\to V_j\subseteq E$ of $G$ with $\eta([t_{j-1},t_j])\subseteq U_j$. Then $$\eta_j:=\varphi_j\circ\eta|_{[t_{j-1},t_j]}\in AC_R([t_{j-1},t_j],E)$$ for all $j\in\{1,\dots,n\}$, and one can thus consider $\eta_j'\in R(I,E)$, which we write as $\eta_j'=[\gamma_j]$ for some $\gamma_j\in\mathcal R(I,E)$. Define
$$\gamma\co I\to TG$$ via $\gamma(t):=T(\varphi_j)\inv(\eta_j(t),\gamma_j(t))$ if $t\in [t_{j-1},t_j[$ with $j\in\{1,\dots,n\}$ and $\gamma(b):=T(\varphi_n)\inv(\eta_n(b),\gamma_n(b))$. We then set
$$\delta(\eta):=[\omega_{\ell}\circ\gamma]\in R(I,\g),$$
where
$$\omega_{\ell}\co TG\to\g, \ v\in T_gG\mapsto g\inv.v:=T_g(\lambda_{g\inv})v\in T_1G.$$
Note that if $G$ and $H$ are smooth Fr\'echet--Lie groups and $f\co G\to H$ is a smooth homomorphism, then $f\circ\eta\in AC_R(I,H)$ for each $\eta\in AC_R(I,G)$ and
\begin{equation}\label{eqn:naturality_delta}
\delta(f\circ\eta)=\Lie(f)\circ\delta(\eta),
\end{equation} 
where $\Lie(f)\circ\delta(\eta):=[\Lie(f)\circ\gamma]$ if $\delta(\eta)=[\gamma]$ (see \cite[5.2(b)]{Gl15}).

\subsection{\texorpdfstring{$R$-regularity (\cite[5.14--5.26]{Gl15})}{R-regularity}}\label{section:RReg}
Let $I=[a,b]\subseteq\RR$ for some $a<b$ and let $E$ be a Fr\'echet space. Let also $V\subseteq E$ be open, and let $g\co I\times V\to E$ be a map. A continuous function $\eta\co I\to E$ is called an \emph{$AC_R$-Carath\'eodory solution} to $y'=g(t,y)$ if $\eta(I)\subseteq V$, the map $t\mapsto g(t,\eta(t))$ is in $R(I,E)$, and
$$\eta(t_2)-\eta(t_1)=\int_{t_1}^{t_2}{g(s,\eta(s))ds} \quad\textrm{for all $t_1,t_2\in I$.}$$

Consider next a smooth manifold $M$ modelled on $E$, and let $(t_0,y_0)\in I\times M$ and $f\co I\times M\to TM$ be a map. Then $\eta\in AC_R(I,M)$ is an \emph{$AC_R$-Carath\'eodory solution} to the initial value problem (IVP)
\begin{equation*}
y'=f(t,y), \quad y(t_0)=y_0
\end{equation*} 
if $\eta(t_0)=y_0$ and for each $t\in I$, there exists $\varepsilon>0$ such that $\eta(I\cap ]t-\varepsilon,t+\varepsilon[)\subseteq U$ for some chart $\varphi\co U\to V\subseteq E$ of $M$ and $\varphi\circ\eta|_{I\cap ]t-\varepsilon,t+\varepsilon[}$ is an $AC_R$-Carath\'eodory solution to $y'=g(t,y)$ with $$g\co\RR\times V\to E, \ (t,y)\mapsto d\varphi(f(t,\varphi\inv(y))).$$

Assume now that $I=[0,1]\subseteq\RR$. If $G$ is a half-Lie group modelled on $E$, with Lie algebra $\g:=T_{1}G$, then $G$ is called \emph{$R$-semiregular} if for each $\gamma\in R(I,\g)$, there exists $\eta\in AC_R(I,G)$ such that
\begin{equation}\label{eqn:IVP_delta}
\delta(\eta)=\gamma\quad\textrm{and}\quad \eta(0)=1_G.
\end{equation}
If it exists, then $\Evol_G(\gamma):=\eta$ is uniquely determined. Note that, writing $\gamma=[\zeta]$ for some $\zeta\in\mathcal R(I,\g)$, the map $\eta\in AC_R(I,G)$ satisfies (\ref{eqn:IVP_delta}) if and only if it is an $AC_R$-Carath\'eodory solution to the IVP 
$$y'=f(t,y),\quad y(0)=1_G$$
with $f\co I\times G\to TG,(t,y)\mapsto y.\zeta(t):=T(\lambda_y)\zeta(t)$. We moreover say that $G$ is \emph{$R$-regular}\footnote{Note that our notion of $R$-regularity is weaker than the one in \cite[5.16]{Gl15}: \emph{$R$-regular} in \emph{loc. cit.} means \emph{$R$-regular with a smooth evolution map} in this paper.} if it is $R$-semiregular and the \emph{evolution map}
$$\Evol_G\co R(I,\g)\to AC_R(I,G)$$
is continuous. Note that if $G$ is $R$-semiregular, then it has an exponential function $\exp_G\co\g\to G$ given by
\begin{equation}\label{eqn:exp_function}
\exp_G(v):=\evol_G(c_v)\quad\textrm{for all $v\in\g$,}
\end{equation}
where $c_v\co I\to\g,t\mapsto v$ is the constant function and 
$$\evol_G\co R(I,\g)\to G, \ \gamma\mapsto \Evol_G(\gamma)(1).$$

If $G$ is a local Lie group modelled on $E$, with Lie algebra $\g:=T_1G$, we call $G$ \emph{locally $R$-semiregular} if there exists an open $0$-neighbourhood $\Omega\subseteq R(I,\g)$ such that for each $\gamma\in\Omega$, there exists $\eta\in AC_R(I,G)$ such that $$\delta(\eta)=\gamma\quad\textrm{and}\quad \eta(0)=1_G.$$ If it exists, then $\Evol_G(\gamma):=\eta$ is uniquely determined. If, moreover, $G$ has a global chart, then $G$ is called \emph{locally $R$-regular} if it is locally $R$-semiregular and $\Omega$ can be chosen such that
$$\Evol_G\co \Omega\to AC_R(I,G)$$
is continuous.

We record for future reference the following results.
\begin{lemma}[{\cite[5.20]{Gl15}}]\label{lemma:5.20}
Set $I=[0,1]\subseteq\RR$, and let $G$ be an $R$-regular Fr\'echet--Lie group with Lie algebra $\g$. Then the evolution map $\Evol_G\co R(I,\g)\to AC_R(I,G)$ is smooth if and only if it is smooth as a map $R(I,\g)\to C(I,G)$. 
\end{lemma}

\begin{lemma}[{\cite[5.25]{Gl15}}]\label{lemma:5.25}
Let $G$ be a Fr\'echet--Lie group. Then 
\begin{enumerate}
\item
$G$ is $R$-semiregular if and only if $G$ is locally $R$-semiregular;
\item
$G$ is $R$-regular with a smooth evolution map if and only if $G$ is locally $R$-regular with a smooth evolution map.
\end{enumerate}
\end{lemma}

\begin{lemma}[{\cite[Theorems A and C]{Gl15}}]\label{lemma:Banach_R_regular}
Every Banach--Lie group is $R$-regular, with a smooth evolution map.
\end{lemma}

\begin{lemma}\label{lemma:naturality_Evol}
Let $G,N$ be Banach Lie groups with respective Lie algebras $\g,\n$, and let $f\co G\to N$ be a Lie group morphism. Then for any $\alpha\in R(I,\g)$, we have
$$f\circ \Evol_G(\alpha)=\Evol_N(\Lie(f)\circ\alpha).$$
\end{lemma}
\begin{proof}
This readily follows from (\ref{eqn:naturality_delta}) in \S\ref{section:LD}. 
\end{proof}

\section{Half-Lie groups}\label{section:STG}
In this section, we introduce in more detail the main protagonists of this paper, namely the half-Lie groups.

\subsection{Definition and examples}
\begin{definition}\label{def:STG}
Let $E$ be a locally convex space and $G$ be a smooth manifold modelled on $E$. We recall from \S\ref{section:subSTG} that $G$ is called a (left) half-Lie group modelled on $E$ if $G$ admits a topological group structure (with respect to the manifold topology) such that all left multiplication maps $\lambda_g\co G\to G$ are smooth. In this paper, we will only consider half-Lie groups modelled on Banach spaces, and simply call them \emph{half-Lie groups}. 
\end{definition}

\begin{remark}
Note that one may replace left multiplication maps by right multiplication maps in Definition~\ref{def:STG}, to obtain a concept of \emph{right half-Lie group}. If $G$ is a right half-Lie group, one can also define the concepts presented in \S\ref{section:LD} and \S\ref{section:RReg}, using right logarithmic derivatives. Of course, if $m$ denotes the multiplication in $G$, then the group $G^{\mathrm{op}}$ obtained by equipping $G$ with the multiplication $m^{\mathrm{op}}(g,h):=m(h,g)$ instead of $m$ is a (left) half-Lie group.
\end{remark}

There are several sources of examples of interest of half-Lie groups, and we now list some of them. 

\begin{example}\label{example:semidirect}
Let $G,N$ be Banach--Lie groups, and let $\pi\co G\to \Aut(N)$ be a continuous automorphic action of $G$ on $N$, in the sense that the map 
$$G\times N\to N, \ (g,n)\mapsto \pi(g)n$$
is continuous. Consider the topological group $H:=N\rtimes_{\pi}G$, with multiplication
$$(n_1,g_1)(n_2,g_2)=(n_1\cdot\pi(g_1)n_2,g_1g_2)\quad \forall n_1,n_2\in N, \ g_1,g_2\in G.$$
Then all left multiplication maps $\lambda_{n,g}\co H\to H$ are smooth, whereas the right multiplication maps $\rho_{n,g}\co H\to H$ are in general only continuous. In particular, $H$ is a half-Lie group. 
\end{example}

\begin{example}
\emph{Extended mapping groups} $H:=N\rtimes_{\pi}G$ are concrete classes 
of Example~\ref{example:semidirect}: here $N=C^k(M,K)$ is the Banach--Lie group of 
$C^k$-maps ($k \in \NN$) 
from the compact smooth manifold $M$ to some Banach--Lie group $K$, 
and $G$ is a Banach--Lie group acting smoothly on $M$, which 
 yields a continuous action by $(\pi(g)f)(s) := f(g\inv s).$
An important special case arises for $M = \mathbb{S}^1$ on which the circle group 
$G=\mathbb{T}$ acts by rigid rotations (loop groups). 
\end{example}

\begin{example}
A second particular case of Example~\ref{example:semidirect} is the following: let $G$ be a Banach--Lie group (e.g. $G=\RR$) and consider a continuous representation $\pi\co G\to \GL(X)$ of $G$ on some Banach space $X$. Then the \emph{affine group} $H=X\rtimes_{\pi}G$ is a half-Lie group.
\end{example}

\begin{example}
A third particular case of Example~\ref{example:semidirect} is the following: let $G$ be a Banach--Lie group and let $\pi\co G\to\Aut(\mathcal A)$ be a continous automorphic action of $G$ on a unital Banach algebra $\mathcal A$ (see \cite{BR02}). 
Then the induced action of $G$ on the group $\mathcal A^{\times}$ of units of $\mathcal A$ yields a half-Lie group $H=\mathcal A^{\times}\rtimes_{\pi} G$.
\end{example}

\begin{example}
Let $M$ be a compact manifold and $r\in\NN$. Then the group $H:=\mathrm{Diff}^{r}(M)$ of $C^r$-diffeomorphisms of $M$ is a right half-Lie group (see \cite[\S VI.2]{Omori97}).
\end{example}

\begin{example}
Let $M$ be a compact smooth manifold or $\RR^d$, 
and $k\in\NN$ with $k>\frac{1}{2}\dim(M)+1$. 
Then the group of Sobolev $H^k$-diffeomorphisms of $M$ is a right half-Lie group. Such groups are important because they can be equipped with right-invariant $H^k$-Sobolev Riemannian metrics, turning them into strong Riemannian Hilbert manifolds (see e.g. \cite{BV17}).
\end{example}

\begin{example}
Set $I=[0,1]\subseteq\RR$, and let $G$ be a Fr\'echet--Lie group with Lie algebra $\g$. Replacing the space of regulated functions $R(I,\g)$ by the Lebesgue space $L^p(I,\g)$ for $p\in \NN\cup\{\infty\}$ (resp. by $C^k(I,\g)$ for $k\in \NN_0\cup\{\infty\}$) in \S\ref{section:RReg} yields a concept of $L^p$- (resp. $C^k$-) semiregularity for $G$ (see \cite[5.16]{Gl15}). For $\mathcal E\in\{R,L^p,C^k\}$, let $\Evol_G\co \mathcal E(I,\g)\to AC_{\mathcal E}(I,G)$ denote the corresponding evolution map. 

Assume that $G$ is $\mathcal E$-semiregular with continuous evolution map $\Evol_G$. Then 
$$\Evol_G\co \mathcal E(I,\g)\to AC_{\mathcal E}(I,G)_*:=\{\eta\in AC_{\mathcal E}(I,G) \ | \ \eta(0)=1_G\}$$
is bijective. Keeping the locally convex space structure on $H:=\mathcal E(I,\g)$ and transporting on $H$ the (pointwise) group multiplication from $AC_{\mathcal E}(I,G)_*$ then turns $H$ into a right half-Lie group (see \cite[5.34, 5.38(e)]{Gl15}).
\end{example}

\begin{example}
The authors of the paper \cite{KMR15} (in which the terminology \emph{half-Lie group} is introduced) construct extensions of certain groups of diffeomorphisms of $\RR^n$, which are right half-Lie groups.
\end{example}

\subsection{Some notation}\label{section:notation}

In this paper, we will focus on the study of the class of examples presented in Example~\ref{example:semidirect}.
In the sequel, unless otherwise stated, $G$ and $N$ will denote Banach Lie groups, and $\pi\co G\to \Aut(N)$ a \emph{continuous action} of $G$ on $N$ by automorphisms, in the sense that the action map $$\pi^{\wedge}\co G\times N\to N, \ (g,n)\mapsto \pi^{\wedge}(g,n):=\pi(g)n$$ is continuous. We set $\g:=\Lie(G)$ and $\n:=\Lie(N)$, and we let $\hh:=T_1H=\n\times\g$ denote the tangent space at the identity of the half-Lie group $H:=N\rtimes_{\pi}G$. For each $n\in N$, we denote by $$\pi^n\co G\to N, \ g\mapsto\pi(g)n$$ the orbit map for $\pi$. We also define the derived action and orbit maps $$\dot{\pi}^{\wedge}\co G\times\n\to\n, \ (g,v)\mapsto \Lie(\pi(g))v\quad\textrm{and}\quad \dot{\pi}^v\co G\to\n, \ g\mapsto \Lie(\pi(g))v$$ for each $v\in \n$. 
For each $k\in\NN\cup\{\infty\}$, we let $$N^k:=\{n\in N \ | \ \textrm{$\pi^n$ is a $C^k$-map}\}$$
denote the space of \emph{$C^k$-elements} for the action $\pi$, and
$$\n^k:=\{v\in\n \ | \ \textrm{$\dot{\pi}^v$ is a $C^k$-map}\}$$ 
the space of \emph{$C^k$-vectors} for the linear action $\dot{\pi}\co G\to \GL(\n),g\mapsto \dot{\pi}^{\wedge}(g,\cdot)$. Note that for any $k\in\NN$, one can define a linear map (cf. \cite[Section~4]{Ne10b})
$$\dd\dot{\pi}\co \g\to \Hom(\n^{k+1},\n^k), \ \dd\dot{\pi}(x)v:=d\dot{\pi}^v(1_G,x)=\frac{d}{dt}\Big|_{t=0}\dot{\pi}(\exp_G(tx))v.$$

Finally, we recall from Lemma~\ref{lemma:Banach_R_regular} that $G$ and $N$ are $R$-regular, with smooth evolution maps $$\Evol_G\co R(I,\g)\to AC_R(I,G)\quad\textrm{and}\quad\Evol_N\co R(I,\n)\to AC_R(I,N),$$ where $I:=[0,1]\subseteq\RR$. 

\section{\texorpdfstring{$R$}{R}-regularity of \texorpdfstring{$H$}{H} and consequences}\label{section:Rreg_consequences}
We use, throughout this section, the notation of \S\ref{section:notation}. We establish in this section the $R$-regularity of the half-Lie group $H=N\rtimes_{\pi}G$, and deduce some consequences, notably Trotter formulas for $H$.

\subsection{\texorpdfstring{$R$}{R}-regularity of \texorpdfstring{$H$}{H}}\label{section:RROHH}
We start with a few preparation lemmas.
\begin{lemma}\label{lemma:compact}
Let $\Gamma$ be a Banach--Lie group and let $(\alpha_n)_{n\in\NN}$ be a sequence of functions in $\LLL_{rc}^{\infty}(I,\Gamma)$ converging uniformly to some $\alpha\in\LLL_{rc}^{\infty}(I,\Gamma)$. Let $K_{\alpha}$ (resp. $K_{\alpha_n}$) denote the closure of $\alpha(I)$ (resp. $\alpha_n(I)$) in $\Gamma$. Then $K_\Gamma:=K_{\alpha}\cup\bigcup_{n\in\NN}K_{\alpha_n}$ is a compact subset of $\Gamma$.
\end{lemma}
\begin{proof}
Note first that Banach--Lie groups are first countable and hence metrisable by the 
Birkhoff--Kakutani Theorem. We equip $\Gamma$ with a metric $\dist_\Gamma\co \Gamma\times \Gamma\to\RR$ compatible with the topology.
Let $\mathcal U$ be an open cover of $K_\Gamma$. Then $\mathcal U$ is also an open cover of each of the compact subsets $K_{\alpha}$ and $K_{\alpha_n}$. Let $(U_i)_{1\leq i\leq n}$ be a finite subcover of $K_{\alpha}$. Then $U:=\bigcup_{i=1}^{n}{U_i}$ contains an $\varepsilon$-neighbourhood of $K_{\alpha}$ for some $\varepsilon>0$. 
Since $(\alpha_n)_{n\in\NN}$ converges uniformly to $\alpha$, there is some $N\in\NN$ such that
$$\dist_\Gamma(\alpha_n(t),\alpha(t))<\varepsilon\quad\textrm{for all $t\in I$ and $n\geq N$.}$$
In particular, $K_{\alpha}\cup\bigcup_{n\geq N}K_{\alpha_n}\subseteq U$. Completing $(U_i)_{1\leq i\leq n}$ with a finite subcover of the compact set $\bigcup_{n=1}^{N-1}K_{\alpha_n}$ then yields the desired finite subcover of $K_\Gamma$.
\end{proof}

\begin{lemma}\label{lemma:naturality_exp}
Let $V$ be an open $0$-neighbourhood in $\n$ such that $\exp_V:=\exp_N|_{V}\co V\to \exp_N(V)$ is a diffeomorphism. Let $U_G\subseteq G$ and $V_1\subseteq V$ be identity neighbourhoods such that $\pi(U_G)\exp_N(V_1)\subseteq\exp_N(V)$. Then  
$$(\exp_{V})\inv(\pi(g)\exp_N(v))=\dot{\pi}(g)v \quad\textrm{for all $g\in U_G$ and $v\in V_1$.}$$
\end{lemma}
\begin{proof}
Fix some $g\in U_G$ and $v\in V_1$, and consider for $t\in [0,1]$ the functions
$$L(t):=(\exp_{V})\inv(\pi(g)\exp_N(tv))\quad\textrm{and}\quad R(t):=\dot{\pi}(g)tv.$$
Note first that $\pi(g)\exp_N(tv)\in \exp_N(V)$ for all $t\in [0,1]$, so that $L(t)$ is well-defined. Moreover, since $\pi(g)\exp_N(tv)=\exp_N(\dot{\pi}(g)tv)$ by naturality of the exponential function, we have $L(t)=R(t)$ whenever $R(t)\in V$, which occurs for all sufficiently small values of $t$ (say for $t\in [0,\epsilon]$ for some $\epsilon>0$), as $\dot{\pi}(g)$ acts continuously on $\n$. On the other hand, for all $s,t\in [0,1]$ with $s+t\in [0,1]$, we have the relations $L(s+t)=L(s)+L(t)$ (see \cite[Example~IV.2.4(a)]{NeJap}) and $R(s+t)=R(s)+R(t)$. Hence $L(t)$ and $R(t)$ are (locally) smooth one-parameter subgroups of the Banach--Lie algebra $(\n,+)$, which coincide for $t\in [0,\epsilon]$, and hence also for all $t\in [0,1]$. In particular, $L(1)=R(1)$, as desired.
\end{proof}

\begin{lemma}\label{lemma:continuity_Tpi}
The linear action $\dot{\pi}^{\wedge}\co G\times\n\to\n,(g,v)\mapsto \dot{\pi}(g)v$ is continuous.
\end{lemma}
\begin{proof}
Let $V,V_1$ and $U_G$ be as in Lemma~\ref{lemma:naturality_exp}.
For each $v\in \n$, consider the maps 
$$\pi^{\exp_N(v)}\co G\to N, \ g\mapsto \pi(g)\exp_N(v)\quad\textrm{and}\quad \dot{\pi}^v\co G\to \n, \ g\mapsto \dot{\pi}(g)v.$$
Then $\pi^{\exp_N(v)}$ is continuous by hypothesis, and hence $\dot{\pi}^v=(\exp_{V})\inv\circ \pi^{\exp_N(v)}$ is also continuous on $U_G$ for all $v\in V_1$ by Lemma~\ref{lemma:naturality_exp}.
From the relation $\dot{\pi}^v\circ\lambda_g=\dot{\pi}(g)\circ \dot{\pi}^v$ and the fact that $\dot{\pi}^v(g)$ is continuous linear in $v$ for each $g\in G$, we then deduce that $\dot{\pi}^v$ is continuous on $G$ for all $v\in\n$. Since $G$ is metrisable, this implies together with \cite[Lemma~5.2]{Ne10b} that $\dot{\pi}^{\wedge}$ is continuous, as desired.
\end{proof}

\begin{theorem}\label{thm:Rreg_Glinear}
Let $G,N$ be Banach--Lie groups, and let $\pi\co G\to \Aut(N)$ be a continuous action. Then $H:=N\rtimes_{\pi}G$ is $R$-regular, i.e. for any $[\gamma]\in R(I,\hh)$ there exists a unique $\widetilde{\gamma}=\Evol_H(\gamma)\in AC_R(I,H)$ with 
\begin{equation}\label{eqn:IVP_H}
\delta(\widetilde{\gamma})=[\gamma]\quad\textrm{and}\quad\widetilde{\gamma}(0)=1_H,
\end{equation}
and the map $\Evol_H\co R(I,\hh)\to AC_R(I,H)$ is continuous.
\end{theorem}
\begin{proof}
For regulated functions $\alpha\co I\to\n$ and $\beta\co I\to\g$, we define the map
$$f_{\alpha,\beta}\co I\to\n, \ t\mapsto \dot{\pi}(\Evol_G(\beta)(t))\alpha(t).$$

Fix some equivalence class $[\gamma]\in R(I,\hh)=R(I,\n)\times R(I,\g)$, which we write as
$[\gamma]=([\alpha],[\beta])$ for some regulated functions $\alpha\co I\to \n$ and $\beta\co I\to\g$. 

Recall first that the IVP
\begin{equation}\label{eqn:IVP_beta_G}
\delta(\widetilde{\beta})=[\beta], \quad \widetilde{\beta}(0)=1_G
\end{equation}
has a unique solution $\widetilde{\beta}:=\Evol_G(\beta)\in AC_R(I,G)$ and that the map $\Evol_G\co R(I,\g)\to AC_R(I,G)$ is smooth, hence continuous.

\noindent
$\bullet$ \underline{Claim 1}: \emph{If the sequences $(\alpha_n)_{n\in\NN}\subseteq \mathcal R(I,\n)$ and $(\beta_n)_{n\in\NN}\subseteq \mathcal R(I,\g)$ converge uniformly to $\alpha$ and $\beta$, respectively, as $n\to\infty$, then $f_{\alpha_n,\beta_n}(t)\to f_{\alpha,\beta}(t)$ as $n\to\infty$, uniformly in $t$.}

Indeed, fix some $\varepsilon>0$. Let $K_{\alpha}$ (resp. $K_{\alpha_n}$) denote the closure of $\alpha(I)$ (resp. $\alpha_n(I)$) in $\n$ and set $K_{\beta}:=\Evol_G(\beta)(I)$ and $K_{\beta_n}:=\Evol_G(\beta_n)(I)$. Note that $K_{\beta}$ and each $K_{\beta_n}$ is compact in $G$ since $\Evol_G$ takes values in a space of continuous functions. Lemma~\ref{lemma:compact} then implies that $K_{\n}:=K_{\alpha}\cup\bigcup_{n\in\NN}K_{\alpha_n}$ and $K_G:=K_{\beta}\cup\bigcup_{n\in\NN}K_{\beta_n}$ are compact, and hence $K:=K_G\times K_{\n}$ is a compact subset of $G\times \n$. We equip $G$ with a metric $\dist_G\co G\times G\to\RR$ compatible with the topology. Since the linear action $\dot{\pi}^{\wedge}\co G\times  \n\to \n,(g,v)\mapsto \dot{\pi}(g)v$ is continuous by Lemma~\ref{lemma:continuity_Tpi}, it is uniformly continuous on $K$ by the Heine--Cantor Theorem, and hence there exists some $\delta>0$ such that
\begin{equation*}
\forall g,h\in K_G, v,w\in K_{\n}: \ \dist_G(g,h)< \delta, \norm{v-w}_{\n}<\delta \implies \norm{\dot{\pi}(g)v-\dot{\pi}(h)w}_{\n}< \varepsilon.
\end{equation*}
Note also that $\Evol_G(\beta_n)$ converges uniformly to $\Evol_G(\beta)$ by continuity of $\Evol_G$, and hence there is some $N\in\NN$ such that 
\begin{equation*}
\dist_G(\Evol_G(\beta_n)(t),\Evol_G(\beta)(t))<\delta \quad\textrm{and}\quad \norm{\alpha_n(t)-\alpha(t)}_{\n}<\delta\quad\textrm{for all $t\in I$ and $n\geq N$,}
\end{equation*}
yielding the claim.

\noindent
$\bullet$ \underline{Claim 2}: \emph{$f_{\alpha,\beta}\in \mathcal R(I,\n)$.}

Indeed, assume first that $\alpha$ is continuous. Then $f_{\alpha,\beta}$ is the composition of the continuous maps 
$$I\to G\times \n, \ t\mapsto (\Evol_G(\beta)(t),\alpha(t))\quad\textrm{and}\quad G\times \n\to \n, \ (g,v)\mapsto \dot{\pi}(g)v,$$
and is in particular continuous, hence uniformly continuous, hence regulated. In particular, if $\alpha$ is a step function (hence piecewise continuous), then $f_{\alpha,\beta}$ is also regulated. Finally, if $\alpha$ is regulated, then it is a uniform limit of step functions, so that $f_{\alpha,\beta}$ is a uniform limit of regulated functions by Claim~1, and is therefore regulated as well.

\noindent
$\bullet$ \underline{Claim 3}: \emph{The half-Lie group $H$ has an evolution map given by $$\Evol_H\co R(I,\hh)\to AC_R(I,H), \ ([\alpha],[\beta])\mapsto (\Evol_N(f_{\alpha,\beta}),\Evol_G(\beta)).$$ }

Indeed, a pair $(\widetilde{\alpha},\widetilde{\beta})\in AC_R(I,H)=AC_R(I,N)\times AC_R(I,G)$ is a solution of the IVP (\ref{eqn:IVP_H}) if and only if it is an $AC_R$-Carath\'eodory solution to  
$$y'=f(t,y),\quad y(0)=1_G$$
with $f\co I\times H\to TH$ given for all $t\in I$, $n\in N$ and $g\in G$ by
\begin{equation*}
\begin{aligned}
f(t,n,g):={}& (n,g).\gamma(t)=T(\lambda_{(n,g)})(\alpha(t),\beta(t))\\
={}&\frac{d}{ds}\Big|_0(n,g)(\exp_N(s\alpha(t)),\exp_G(s\beta(t)))\\
={}&\frac{d}{ds}\Big|_0\big(n\cdot\pi(g)\exp_N(s\alpha(t)),g\cdot\exp_G(s\beta(t))\big)\\
={}&\big(n.\dot{\pi}(g)\alpha(t),g.\beta(t)\big).\\
\end{aligned}
\end{equation*}
Hence $(\widetilde{\alpha},\widetilde{\beta})$ is a solution of (\ref{eqn:IVP_H}) if and only if $\widetilde{\beta}$ is a solution of the IVP (\ref{eqn:IVP_beta_G}), so that $\widetilde{\beta}=\Evol_G(\beta)$, and $\widetilde{\alpha}$ is a solution of the IVP
$$\delta(\widetilde{\alpha})=[f_{\alpha,\beta}], \quad \widetilde{\alpha}(0)=1_N,$$
as claimed.

\noindent
$\bullet$ \underline{Claim 4}: \emph{The map $\Evol_H$ is continuous.} 

Indeed, since $\Evol_H$ is the composition of the maps $$R(I,\n)\times R(I,\g)\to R(I,\n)\times R(I,\g), \ ([\alpha],[\beta])\mapsto ([f_{\alpha,\beta}],[\beta])$$
and $$R(I,\n)\times R(I,\g)\to AC_R(I,N)\times AC_R(I,G), \ ([\alpha],[\beta])\mapsto (\Evol_N(\alpha),\Evol_G(\beta))$$ by Claim~3, it suffices to show that the map $R(I,\n)\times R(I,\g)\to R(I,\n),([\alpha],[\beta])\mapsto [f_{\alpha,\beta}]$ is continuous, which follows from Claim~1.
\end{proof}

\begin{corollary}
The half-Lie group $H$ has a continuous exponential function 
$$\exp_H\co\n\times\g\to N\rtimes_{\pi}G, \ (v,x)\mapsto \big(\evol_N(f_{v,x}),\exp_G(x)\big),$$
where $$f_{v,x}\co I\to \n, \ t\mapsto \dot{\pi}(\exp_G(tx))v.$$ 
\end{corollary}
\begin{proof}
This readily follows from (\ref{eqn:exp_function}) together with Claim~3 in the proof of Theorem~\ref{thm:Rreg_Glinear}.
\end{proof}

\begin{remark}
In the notation of (the proof of) Theorem~\ref{thm:Rreg_Glinear}, the map $\Evol_H\co R(I,\hh)\to AC_R(I,H)$ is $C^k$ in some neighbourhood of $([\alpha],[\beta])$ if and only if the map $R(I,\n)\times R(I,\g)\to R(I,\n),([\alpha],[\beta])\mapsto [f_{\alpha,\beta}]$ is $C^k$ in some neighbourhood of $([\alpha],[\beta])$ (this is because the logarithmic derivative in the Banach--Lie groups $G$ and $N$ is a smooth inverse for the evolution map, see \cite[5.29]{Gl15}). In particular, $\Evol_H$ is in general \emph{not} a smooth map. For instance, if $\alpha=c_v$ ($v\in\n$) and $\beta=c_x$ ($x\in\g$) are constant functions, then $f_{\alpha,\beta}=f_{v,x}\co I\to \n,t\mapsto \dot{\pi}(\exp_G(tx))v$ will in general only be $C^k$ in $x$ provided that $v$ is a $C^k$-vector for $\dot{\pi}$. 
\end{remark}

\subsection{Trotter formulas for \texorpdfstring{$H$}{H}}\label{section:TFFH}
Now that the $R$-regularity of $H$ is established, we investigate its consequences on the structure of the space $\Hom^1(\RR,H)$ of $C^1$ one-parameter subgroups of $H$, and its relations to the tangent space $\hh=T_1H$ of $H$ at $1_H$. We start with an easy observation, which follows from the existence of an exponential function for $H$.
\begin{lemma}
The map $$\Hom^1(\RR,H)\to\hh, \ \gamma\mapsto\gamma'(0):=T_0(\gamma)(1)$$ is bijective, with inverse given by $\hh\to \Hom^1(\RR,H), v\mapsto [t\mapsto\exp_H(tv)]$.
\end{lemma}
Note that, since the multiplication in $H$ is only continuous, one cannot a priori define a Lie algebra structure on $\hh=\n\times\g$ using left invariant vector fields as in \S\ref{section:LG_NeJap}. Indeed, the Lie bracket $[(v_1,x_1),(v_2,x_2)]:=([v_1,v_2]+\dd\dot{\pi}(x_1)v_2-\dd\dot{\pi}(x_2)v_1,[x_1,x_2])$ in $\n\times\g$ only makes sense in general when $v_1,v_2$ are $C^1$-vectors for $\dot{\pi}$.

On the other hand, for a Lie group $\Gamma$ with Lie algebra $\Lie(\Gamma)$ and smooth exponential function $\exp_{\Gamma}\co \Lie(\Gamma)\to\Gamma$, there may be an alternative approach to describe the Lie algebra structure on $\Lie(\Gamma)$, namely when $\Gamma$ has the Trotter property (resp. commutator property): in this case, one may define a vector space structure (resp. a Lie bracket) on $\Hom^1(\RR,\Gamma)$ by the Trotter formula (resp. commutator formula), and then transport the obtained Lie algebra structure on $\Hom^1(\RR,\Gamma)$ using the identification $\Hom^1(\RR,\Gamma)\approx \Lie(\Gamma)$. We recall that $\Gamma$ is said to have the \emph{Trotter property} if for all $\gamma_1,\gamma_2\in\Hom^1(\RR,\Gamma)$,
\begin{equation}\label{eqn:TrF}
\lim_{n\to\infty}\big(\gamma_1(\tfrac{t}{n})\gamma_2(\tfrac{t}{n})\big)^n=\exp_{\Gamma}\big(t(\gamma_1'(0)+\gamma_2'(0))\big)
\end{equation}
uniformly in $t$ on compact subsets of $\RR$, and the \emph{commutator property} if for all $\gamma_1,\gamma_2\in\Hom^1(\RR,\Gamma)$,
\begin{equation}\label{eqn:CommF}
\lim_{n\to\infty}\big(\gamma_1(\tfrac{\sqrt{t}}{n})\gamma_2(\tfrac{\sqrt{t}}{n})\gamma_1(-\tfrac{\sqrt{t}}{n})\gamma_2(-\tfrac{\sqrt{t}}{n})\big)^{n^2}=\exp_{\Gamma}\big(t[\gamma_1'(0),\gamma_2'(0)]\big)
\end{equation}
uniformly in $t$ on compact subsets of $[0,\infty[$. Note that $\Gamma$ has both the Trotter and commutator properties as soon as $\Gamma$ has the \emph{strong Trotter property}, that is, as soon as for each $C^1$-curve $\gamma\co I\to \Gamma$ with $\gamma(0)=1_{\Gamma}$, 
\begin{equation}\label{eqn:STF}
\lim_{n\to\infty}\gamma(\tfrac{t}{n})^n=\exp_{\Gamma}(t\gamma'(0))
\end{equation}
uniformly in $t$ on compact subsets of $[0,\infty[$ (see \cite[Theorem~H]{Gl15}).

It is now natural to investigate which of the above properties are satisfied in our context, namely by the half-Lie group $H$. As suggested above, one cannot expect in general $H$ to have the commutator property without further assumption on the curves $\gamma_1,\gamma_2\in\Hom^1(\RR,H)$ involved in the commutator formula (\ref{eqn:CommF}). We illustrate this with the following example.

\begin{example}
Fix some $(\lambda_n)_{n\in\NN}\in\CC^{\NN}$. Consider the Hilbert space $V=\ell^2(\CC)$, together with the continuous linear $\RR$-action $$\pi\co \RR\to \GL(V), \ \pi(t)(x_n)_{n\in\NN}:= (e^{i\lambda_nt}x_n)_{n\in\NN}.$$ 
Then $V\rtimes_{\pi}\RR$ is a half-Lie group. Note that a vector $v=(x_n)_{n\in\NN}\in V$ is a $C^1$-vector for $\pi$ if and only if $\sum_{n\in\NN}{|\lambda_nx_n|^2}<\infty$. Consider now the smooth one-parameter subgroups $\gamma_1(t)=(tv,0)$ and $\gamma_2(t)=(0,t)$ of $V\rtimes_{\pi}\RR$, where $v=(x_n)_{n\in\NN}\in V$. Then
$$\gamma(t):=\gamma_1(t)\gamma_2(t)\gamma_1(t)\inv \gamma_2(t)\inv=\big(t\cdot ((1-e^{i\lambda_nt})x_n)_{n\in\NN},0\big),$$
and hence
$$\gamma(\sqrt{t}/m)^{m^2}=\big(m\sqrt{t}\cdot ((1-e^{i\lambda_n\sqrt{t}m\inv})x_n)_{n\in\NN},0\big).$$
Since for each $n\in\NN$,
$$\lim_{m\to\infty}{m\sqrt{t}\cdot (1-e^{i\lambda_n\sqrt{t}m\inv})x_n}=-i\lambda_nx_nt,$$
we deduce that $\lim_{m\to\infty}\gamma(\sqrt{t}/m)^{m^2}$ exists in $V$ if and only if $\sum_{n\in\NN}{|\lambda_nx_n|^2}<\infty$, that is, if and only if $v$ is a $C^1$-vector for $\pi$. In particular, if $(\lambda_n)_{n\in\NN}\notin\ell^{\infty}(\CC)$, then there exists a $v\in V\setminus V^1$, so that for the above choices of one-parameter subgroups $\gamma_1,\gamma_2$, the limit in (\ref{eqn:CommF}) does not exist.
\end{example}

As noticed in \cite[Theorem~I]{Gl15}, the $R$-regularity of a Fr\'echet--Lie group implies the strong Trotter formula. As it turns out, the proof of Theorem~I in \emph{loc. cit.} (see \cite[Section~12]{Gl15}) can be easily adapted in the setting of half-Lie groups. Combined with Theorem~\ref{thm:Rreg_Glinear}, this will then imply that $H$ has the strong Trotter property (\ref{eqn:STF}). 

\begin{prop}\label{prop:basiclimit}
Let $H$ be an $R$-regular half-Lie group. Then $H$ has the strong Trotter property: for all $\zeta\in C^1(I,H)$ with $\zeta(0)=1_H$, $$\lim_{n\to\infty}{\zeta(t/n)^n}=\exp_H(t\zeta'(0))$$
uniformly in $t$ on compact subsets of $[0,+ \infty[$.
\end{prop}
\begin{proof}
Let $m\in\NN$. For each $n\geq m$, define
$$\zeta_n\co [0,m]\to H, \ t\mapsto (\zeta(t/n))^n.$$
We have to prove that 
$$\zeta_n(t)\stackrel{n\to\infty}{\longrightarrow}\exp_H(t\zeta'(0))\quad\textrm{uniformly in $t\in [0,m]$.}$$
In other words, fixing an open identity neighbourhood $U$ in $H$, we have to show (cf. Remark~\ref{remark:leftright_uniform} below) that
\begin{equation}\label{eqn:TP1}
(\exists n_0\geq m)(\forall n\geq n_0)(\forall t\in [0,m]) \quad \zeta_n(t)\in \exp_H(t\zeta'(0))U.
\end{equation}

Consider the map $\evol\co \mathcal R(I,\hh)\to H,\sigma\mapsto\Evol_H(\sigma)(1)$. For $v\in\hh$, consider also the constant curve
$$c_v\co I\to\hh, \ s\mapsto v.$$
Since the map $\hh\to \mathcal R(I,\hh),v\mapsto c_v$ is continuous, the set
$K:=\{c_{t\zeta'(0)} \ | \ t\in [0,m]\}$
is compact.

\smallskip
\noindent
$\bullet$ \underline{Claim 1}: \emph{There is an open $0$-neighbourhood $Q\subseteq \mathcal R(I,\hh)$ such that $\evol(\theta+Q)\subseteq \evol(\theta)U$ for every $\theta\in K$.}

Indeed, consider the continuous function $f\co \Rcal(I,\hh)\times\Rcal(I,\hh)\to H,(\theta_1,\theta_2)\mapsto \evol(\theta_1)\inv\evol(\theta_2)$. Since $f\inv(U)$ is open and contains the compact set $\Delta_K:=\{(\theta,\theta) \ | \ \theta\in K\}$, there exists some open $0$-neighbourhood $Q\subseteq \Rcal(I,\hh)$ such that $\Delta_K+(Q\times Q)\subseteq f\inv(U)$. 
In particular, $f(\Delta_K+(\{0\}\times Q))\subseteq U$, yielding the claim.

\smallskip

For $n\geq m$ and $t\in [0,m]$, we consider the continuous curve $\alpha_{n,t}\co I\to H$, defined piecewise by
$$\alpha_{n,t}(s):=\zeta(t/n)^k\zeta((s-k/n)t)\quad\textrm{for $s\in [k/n,(k+1)/n]$, $k=0,\dots,n-1$.}$$
Then $\alpha_{n,t}$ is piecewise $C^1$, so that $\alpha_{n,t}\in AC_{R}(I,H)$ with $\beta_{n,t}:=\delta(\alpha_{n,t})\in \mathcal R(I,\hh)$ (here we make a slight abuse of notation and identify the logarithmic derivative of a function with some representative of the equivalence class it defines). Moreover,
$$\zeta_{n}(t)=\alpha_{n,t}(1)=\evol(\beta_{n,t})$$ and
$$\beta_{n,t}(s)=t\delta(\zeta)((s-k/n)t)\quad\textrm{for $s\in ]k/n,(k+1)/n[$, $k=0,\dots,n-1$.}$$

\smallskip
\noindent
$\bullet$ \underline{Claim 2}: \emph{There exists some $n_0\geq m$ such that $\beta_{n,t}-c_{t\zeta'(0)}\in Q$ for all $n\geq n_0$ and $t\in [0,m]$.}

Indeed, let $\epsilon>0$ be such that $$\overline{B}_{\epsilon}(0):=\{\tau\in \mathcal R(I,\hh) \ | \ \norm{\tau}_{\mathcal L^{\infty}}\leq\epsilon\}\subseteq Q.$$
Since $\delta(\zeta)\co I\to\hh$ is continuous, there is some $\varepsilon\in ]0,1]$ such that 
$$\norm{\delta(\zeta)(x)-\delta(\zeta)(0)}_{\hh}\leq \frac{\epsilon}{m}\quad\textrm{for all $x\in [0,\varepsilon]$.}$$
Let $n_0\geq m$ be such that $\tfrac{m}{n_0}\leq\varepsilon$, and let $n\geq n_0$. Then 
$$(s-\tfrac{k}{n})t\leq\tfrac{m}{n_0}\leq \varepsilon\quad\textrm{for all $t\in [0,m]$ and $s\in [k/n,(k+1)/n]$, $k=0,\dots,n-1$.}$$
Hence
$$\beta_{n,t}(\cdot)=t\delta(\zeta)(0)+t\big(\delta(\zeta)((\cdot-k/n)t)-\delta(\zeta)(0)\big)\in t\delta(\zeta)(0)+t\overline{B}_{\epsilon/m}(0)\subseteq t\delta(\zeta)(0)+\overline{B}_{\epsilon}(0)$$
for all $n\geq n_0$ and $t\in [0,m]$. Since $\delta(\zeta)(0)=\zeta'(0)$, this implies that
$$\beta_{n,t}-c_{t\zeta'(0)}\in \overline{B}_{\epsilon}(0)\subseteq Q$$
for all $n\geq n_0$ and $t\in [0,m]$, as desired.

\smallskip

It follows from Claims 1 and 2 that for all $n\geq n_0$ and $t\in [0,m]$,
$$\zeta_n(t)=\evol(\beta_{n,t})\in \evol(c_{t\zeta'(0)}+Q)\subseteq \evol(c_{t\zeta'(0)})U=\exp_H(t\zeta'(0))U,$$
proving (\ref{eqn:TP1}).
\end{proof}

\begin{remark}\label{remark:leftright_uniform}
Given a topological group $\Gamma$ and continuous curves $\gamma_n,\gamma\co\RR\to\Gamma$ ($n\in\NN$), the following assertions are equivalent 
(see \cite[Lemma~A.5.21]{GlNe16}):
\begin{enumerate}
\item
$\lim_{n\to\infty}\gamma_n(t)\gamma(t)\inv=1_{\Gamma}$ uniformly in $t$ on compact subsets.
\item
$\lim_{n\to\infty}\gamma(t)\inv\gamma_n(t)=1_{\Gamma}$ uniformly in $t$ on compact subsets.
\item
$\gamma_n\to\gamma$ as $n\to\infty$ in the compact open topology.
\end{enumerate}
\end{remark}

\begin{corollary}\label{cor:STP}
Let $G,N$ be Banach--Lie groups, and let $\pi\co G\to \Aut(N)$ be a continuous action. Then $H:=N\rtimes_{\pi}G$ has the strong Trotter property.
\end{corollary}
\begin{proof}
This follows from Theorem~\ref{thm:Rreg_Glinear} and Proposition~\ref{prop:basiclimit}.
\end{proof}

In contrast to the case of Lie groups, given $\gamma_1,\gamma_2\in\Hom^1(\RR,H)$, one cannot immediately apply Corollary~\ref{cor:STP} to the curve $\gamma(t):=\gamma_1(t)\gamma_2(t)$ in order to deduce the Trotter formula (\ref{eqn:TrF}) in $H$, as $\gamma$ might not be $C^1$. Nevertheless, with some extra work, one can prove a strengthening of (\ref{eqn:TrF}) for $H$: namely, the formula (\ref{eqn:TrF}) holds in $H$ for any $C^1$-curves $\gamma_1,\gamma_2\co I\to H$ with $\gamma_1(0)=\gamma_2(0)=1_H$, under some additional mild condition on $(\gamma_1,\gamma_2)$ (which includes the case where $\gamma_1,\gamma_2\in\Hom^1(\RR,H)$). 

\begin{definition}
For $i=1,2$, let $\gamma_i=(\alpha_i,\beta_i)\co I\to H=N\rtimes_{\pi}G$ be a $C^1$-curve such that $\gamma_i(0)=1_H$. We say that the couple $(\gamma_1,\gamma_2)$ \emph{has the property $(\star)$} if
\begin{enumerate}
\item[($\star$)] There exists some $C^1$-curve $\alpha_3\co I\to N$ with $\alpha_3(0)=1_N$ such that the map $$h\co I\to N, \ t\mapsto \pi(\beta_1(t))\alpha_3(t)$$ is $C^1$ and such that $\alpha_3'(0)=h'(0)=\alpha_2'(0)$.
\end{enumerate}
\end{definition}

\begin{lemma}\label{lemma:property_star}
For $i=1,2$, let $\gamma_i=(\alpha_i,\beta_i)\co I\to H=N\rtimes_{\pi}G$ be a $C^1$-curve such that $\gamma_i(0)=1_H$. Then $(\gamma_1,\gamma_2)$ has the property $(\star)$ as soon as one of the following conditions holds:
\begin{enumerate}
\item
$\beta_1\co I\to G$ is a (local) one-parameter subgroup.
\item
$\alpha_2'(0)\in\n^1$.
\end{enumerate}
\end{lemma}
\begin{proof}
If (1) holds, we define $\alpha_3\co I\to N$ by $(\alpha_3(t),\beta_1(t))=\exp_H(t(\alpha_2'(0),\beta_1'(0)))$. Then for all $t\in I$,
$$(\alpha_3(2t),\beta_1(2t))=(\alpha_3(t),\beta_1(t))^2=(\alpha_3(t)\cdot\pi(\beta_1(t))\alpha_3(t),\beta_1(t)^2)$$
and hence $$\pi(\beta_1(t))\alpha_3(t)=\alpha_3(t)\inv\alpha_3(2t),$$ so that $\alpha_3$ indeed has all the desired properties.

If (2) holds, we define $\alpha_3\co I\to N$ by $\alpha_3(t)=\exp_N(t\alpha_2'(0))$. Then $$\pi(\beta_1(t))\alpha_3(t)=\exp_N(t\cdot \dot{\pi}(\beta_1(t))\alpha_2'(0))\quad\textrm{for all $t\in I$},$$ so that $\alpha_3$ also has all the desired properties in this case.
\end{proof}

\begin{theorem}\label{thm:TFimproved}
Let $G,N$ be Banach Lie groups and $\pi\co G\to \Aut(N)$ be a continuous action. Let $\gamma_1,\gamma_2\co I\to H$ be $C^1$-curves with $\gamma_1(0)=\gamma_2(0)=1_H$, and assume that $(\gamma_1,\gamma_2)$ has the property $(\star)$. Then
\begin{equation*}
\lim_{n\to\infty}{\big(\gamma_1(\tfrac{t}{n})\gamma_2(\tfrac{t}{n})\big)^n}=\exp_H\big(t(\gamma_1'(0)+\gamma_2'(0))\big)
\end{equation*}
uniformly in $t$ on compact subsets of $[0,\infty[$.
\end{theorem}
\begin{proof}
We fix some left invariant metric $\dist$ on $G$ compatible with the topology, and we normalise the norm on $\n$ so that $\norm{[x,y]}\leq \norm{x}\cdot\norm{y}$ for all $x,y\in\n$.
We also fix some open connected symmetric $0$-neighbourhood $V$ in $\n$ such that $\exp_{V}:=\exp_N|_{V}$ is a diffeomorphism onto the open identity neighbourhood $U:=\exp_N(V)\subseteq N$. Finally, we fix some convex open $0$-neighbourhood $V_1\subseteq V$ in $\n$ such that $\exp_V(V_1)^2\subseteq U$, so that the local multiplication $x*y:=\exp_V\inv(\exp_V(x)\exp_V(y))$ is defined for all $x,y\in V_1$.

Given a curve $\alpha\co [0,r]\to U$ for some $r\in\RR_+$, we put $\overline{\alpha}:=\exp_V\inv\circ\alpha\co [0,r]\to V$. Note that $\alpha$ is continous (resp. $C^1$) if and only if $\overline{\alpha}$ is continuous (resp. $C^1$). We also recall that 
$$\pi(g)\exp_N(v)=\exp_N(\dot{\pi}(g)v)\quad\forall g\in G, \ v\in\n.$$

For $i=1,2$, write $\gamma_i(t)=(\alpha_i(t),\beta_i(t))$ for some $C^1$-curves $\alpha_i\co I\to N$ and $\beta_i\co I\to G$.
By assumption, there exists some $C^1$-curve $\alpha_3\co I\to N$ such that the map $h\co I\to N,t\mapsto \pi(\beta_1(t))\alpha_3(t)$ is $C^1$ and such that $\alpha_3(0)=1_N$ and $\alpha_3'(0)=h'(0)=\alpha_2'(0)$.

Consider the continuous curves $c_1,c_2\co I\to H$, $x_1,x_2\co I\to N$ and $y\co I\to G$ defined by
$$c_1(t):=(x_1(t),y(t)):=\gamma_1(t)\gamma_2(t)=(\alpha_1(t)\cdot \pi(\beta_1(t))\alpha_2(t),\beta_1(t)\beta_2(t))$$
and
$$c_2(t):=(x_2(t),y(t)):=(\alpha_1(t)\cdot \pi(\beta_1(t))\alpha_3(t),\beta_1(t)\beta_2(t)).$$
Thus $c_2$ is a $C^1$-curve with $c_2'(0)=(\alpha_1'(0)+\alpha_2'(0),\beta_1'(0)+\beta_2'(0))=\gamma_1'(0)+\gamma_2'(0)$. Corollary~\ref{cor:STP} then implies that 
\begin{equation}\label{eqn:c2_Trotter}
\lim_{n\to\infty}{c_2(t/n)^n}=\exp_H\big(tc_2'(0)\big)=\exp_H\big(t(\gamma_1'(0)+\gamma_2'(0))\big)
\end{equation}
uniformly in $t$ on compact subsets of $[0,\infty[$. It thus remains to show that
\begin{equation*}
\lim_{n\to\infty}{c_1(t/n)^{-n}\cdot c_2(t/n)^n}=1_H
\end{equation*}
uniformly in $t$ on compact subsets of $[0,\infty[$.

Consider the continuous curves $x\co I\to N$ and $z_s\co I\to N$ for each $s\in\NN$, defined for $t\in I$ by
$$x(t):=x_2(t)x_1(t)\inv\quad\textrm{and}\quad (z_s(t),y(t)^s):=c_2(t)^s=(x_2(t),y(t))^s=\Big(\prod_{i=0}^{s-1}{\pi(y(t)^{i})x_2(t)},y(t)^s\Big).$$

One then easily computes that
$$c_1^{-n}=(x_1,y)^{-n}=\Big(\prod_{s=1}^{n}{\pi(y^{-s})x_1\inv},y^{-n}\Big)\quad\textrm{and}\quad c_2^n=(x_2,y)^n=\Big(\pi(y)^n\prod_{s=n}^{1}{\pi(y^{-s})x_2},y^n\Big),$$
so that
$$c_1^{-n}c_2^n=\bigg(\Big(\prod_{s=1}^{n}{\pi(y^{-s})x_1\inv}\Big)\cdot \Big(\prod_{s=n}^{1}{\pi(y^{-s})x_2}\Big),1_G\bigg)=\bigg(\prod_{s=1}^{n}{\pi(y^{-s})(z_s\inv\cdot x\cdot z_s)},1_G\bigg).$$
Indeed, setting $\Pi^{L}_n:=\prod_{s=1}^{n}{\pi(y^{-s})x_1\inv}$, $\Pi^{R}_n:=\prod_{s=n}^{1}{\pi(y^{-s})x_2}$ and $\Pi^{LR}_n:=\prod_{s=1}^{n}{\pi(y^{-s})(z_s\inv\cdot x\cdot z_s)}$, we deduce from the equalities $$\Pi^{R}_n=\pi(y^{-n})z_n\quad\textrm{and}\quad \pi(y^{-n})z_n\cdot \pi(y^{-(n+1)})z_{n+1}\inv=\pi(y^{-(n+1)})x_2\inv$$
that $$\Pi^{L}_n\cdot\Pi^{R}_n\cdot \pi(y^{-(n+1)})(z_{n+1}\inv x z_{n+1})=\Pi^{L}_n\cdot \pi(y^{-(n+1)})x_1\inv\cdot \pi(y^{-(n+1)})z_{n+1}=\Pi^{L}_{n+1}\cdot\Pi^{R}_{n+1},$$
so that the claim follows by induction on $n$.

Fixing some $R\in\RR_+$, we thus have to show that 
$$F(n,t):=\prod_{s=1}^{n}{\pi(y(\tfrac{t}{n})^{-s})(z_s(\tfrac{t}{n})\inv\cdot x(\tfrac{t}{n})\cdot z_s(\tfrac{t}{n}))}$$
converges uniformly for $t\in[0,R]$ to $1_N$ as $n\to\infty$. For each $s\in\{1,\dots,n\}$, we set
$$F_s(n,t):=\pi(y(\tfrac{t}{n})^{-s})(z_s(\tfrac{t}{n})\inv\cdot x(\tfrac{t}{n})\cdot z_s(\tfrac{t}{n})).$$

Since the local multiplication $f\co V_1\times V_1\to V,(x,y)\mapsto x*y$ is smooth, its second differential 
$$d^2f\co (V_1\times V_1)\times (\n\times\n)\times (\n\times\n)\to\n$$
is continous, where $\n\times\n$ is endowed with the norm $(x,y)\mapsto \norm{x}+\norm{y}$. Hence there exists some $r>0$ with $\overline{B}^{\n}_{r}(0)\subseteq V_1$ such that
$$d^2f((\overline{B}^{\n\times\n}_{r}(0))^3)\subseteq \overline{B}_1^{\n}(0),$$
where $\overline{B}^X_r(0)$ denotes the closed ball centered at $0$ and of radius $r$ in the metric space $X$. In particular, the continuous bilinear map $d^2f(x,y,\cdot)\co (\n\times\n)^2\to\n$ has operator norm
\begin{equation}
\norm{d^2f(x,y,\cdot)}_{op}\leq \frac{1}{r^2}\quad\textrm{for all $x,y\in V_1$ with $\norm{x},\norm{y}\leq r$.}
\end{equation}

\smallskip
\noindent
$\bullet$ \underline{Claim 1.1}: \emph{There is a constant $C=C(r)$ such that $$\norm{x*y}\leq \norm{x+y}+C\tfrac{\norm{x}^2+\norm{y}^2}{2}\quad\textrm{for any $x,y\in V_1$ with $\norm{x},\norm{y}\leq r$.}$$ }
\\
Indeed, the first order Taylor expansion of $f$ around $(0,0)$ with remainder term (see \cite[Prop. I.2.3(v)]{NeJap}) yields
\begin{equation*}
\begin{aligned}
\norm{x*y-x-y}&\leq \int_0^1{(1-t)\norm{d^2f((tx,ty),(x,y),(x,y))}dt}\\
&\leq \int_0^1{(1-t)\norm{d^2f(tx,ty,\cdot)}_{op}(\norm{x}+\norm{y})^2dt}\\
&\leq \frac{\norm{x}^2+\norm{y}^2}{r^2}.
\end{aligned}
\end{equation*}

\smallskip
\noindent
$\bullet$ \underline{Claim 1.2}: \emph{For each $n\in\NN$ and $\rho>0$, set $a_n(\rho):=\frac{1}{C}((C\rho+1)^n-1)$. Then $$a_{2n}(\rho)=2a_n(\rho)+Ca_n(\rho)^2\quad\textrm{and}\quad a_n(\rho/n)\leq a(\rho):=\frac{1}{C}(e^{C\rho}-1)\quad\textrm{for all $n\in\NN$ and $\rho>0$.}$$}
\\
Indeed, the equality is an easy computation, while the inequality follows from the fact that the sequence $(1+C\rho/n)^n$ is increasing and converging to $e^{C\rho}$.

\smallskip
\noindent
$\bullet$ \underline{Claim 1.3}: \emph{Let $n=2^k$ for some $k\in\NN$. Then for all $\rho>0$ with $a(\rho)\leq r$ and all $x_1,\dots,x_{n}\in V_1$ with $\norm{x_i}\leq\rho/n$ ($i=1,\dots,n$), the product $x_1*x_2*\dots *x_{n}$ is defined and we have $$\norm{x_1*x_2*\dots *x_{n}}\leq a_n(\rho/n).$$}
\\
Indeed, we prove the claim by induction on $k$. If $k=0$, there is nothing to prove. Assume now that the claim is true for $n=2^k$, and let us prove it for $2n$. Let thus $x_1,\dots,x_{2n}\in V_1$ with $\norm{x_i}\leq \tfrac{\rho}{2n}$ for all $i$. Note first that by induction hypothesis and Claim~1.2, we have $$\norm{x_1*\dots *x_{n}}\leq a_n(\rho/n)\leq a(\rho)\leq r$$ and similarly $\norm{x_{n+1}*\dots *x_{2n}}\leq r$, so that $x_1*\dots *x_{n}$ and $x_{n+1}*\dots *x_{2n}$ belong to $V_1$. Hence their product $f(x_1*\dots *x_{n},x_{n+1}*\dots *x_{2n})$ is defined. Moreover, the induction hypothesis (with $\rho/2$ instead of $\rho$) and Claim~1.1 imply that
\begin{equation*}
\begin{aligned}
\norm{x_1*\dots *x_{2n}}&\leq \norm{x_1*\dots *x_{n}}+\norm{x_{n+1}*\dots *x_{2n}}+C\frac{\norm{x_1*\dots *x_{n}}^2+\norm{x_{n+1}*\dots *x_{2n}}^2}{2}\\
&\leq 2a_n(\tfrac{\rho}{2n})+Ca_n(\tfrac{\rho}{2n})^2=a_{2n}(\tfrac{\rho}{2n}),
\end{aligned}
\end{equation*}
where the last equality follows from Claim~1.2.

\smallskip
\noindent
$\bullet$ \underline{Claim 1.4}: \emph{Let $x,y\in\n$. Then $\exp_N(x)\exp_N(y)\exp_N(x)\inv=\exp_N(e^{\ad x}y)$ and $\norm{e^{\ad x}y}\leq e^{\norm{x}}\norm{y}$.}
\\
Indeed, by naturality of the exponential map, we have (see \cite[(2.5.4) p.340]{NeJap})
\begin{equation*}
\exp_N(x)\exp_N(y)\exp_N(x)\inv =\exp_N\circ \mathrm{Ad}(\exp_N(x))(y)=\exp_N(e^{\ad x}y).
\end{equation*}
Moreover, since $\norm{[x,z]}\leq \norm{x}\cdot\norm{z}$ for all $z\in\n$, so that $\ad x$ has operator norm $\norm{\ad x}\leq \norm{x}$, we have  
$$\norm{e^{\ad x}y}\leq \norm{e^{\ad x}}\cdot\norm{y} = e^{\norm{\ad x}}\norm{y}\leq e^{\norm{x}}\norm{y}.$$

\smallskip
\noindent
$\bullet$ \underline{Claim 2.1}: \emph{Let $\epsilon_1>0$. Then there is some $N_1\in\NN$ such that the curve $\overline{x}:=\exp_V\inv\circ x$ satisfies $$x(\tfrac{t}{n})\in U\quad\textrm{and}\quad\norm{\overline{x}(\tfrac{t}{n})}\leq \epsilon_1/n\quad\textrm{for all $t\in [0,R]$ and $n\geq N_1$.}$$}
\\
Indeed, since $x(t)=\alpha_1(t)\cdot \pi(\beta_1(t))(\alpha_3(t)\alpha_2(t)\inv)\cdot \alpha_1(t)\inv$ is continuous, certainly $x(\tfrac{t}{n})\in U$ for all large enough $n$ (with $t\in [0,R]$). We may thus define $\overline{x}(\tfrac{t}{n})=\exp_V\inv(x(\tfrac{t}{n}))$ for $n$ large enough. 

Let $\epsilon_1'>0$. Recall that $\alpha_2'(0)=\alpha_3'(0)$. Hence there is some $N_1'\in\NN$ such that for all $t\in [0,R]$ and $n\geq N_1'$, we have
$$\Bignorm{\frac{\overline{\alpha}_3(\tfrac{t}{n})-\overline{\alpha}_2(\tfrac{t}{n})}{t/n}}\leq \epsilon_1',\quad \Bignorm{\frac{\overline{\alpha}_2(\tfrac{t}{n}))}{t/n}}^2\leq \norm{\alpha_2'(0)}^2+\epsilon_1'\quad \textrm{and}\quad \Bignorm{\frac{\overline{\alpha}_3(\tfrac{t}{n}))}{t/n}}^2\leq \norm{\alpha_2'(0)}^2+\epsilon_1'.$$ 
Up to increasing $N_1'$, we may further assume that $\frac{CR(\norm{\alpha_2'(0)}^2+\epsilon_1')}{n}\leq \epsilon_1'$ for all $n\geq N_1'$.
It then follows from Claim~1.1 that for all $t\in [0,R]$ and $n\geq N_1'$,
\begin{equation*}
\begin{aligned}
\norm{\exp_V\inv(\alpha_3(\tfrac{t}{n})\alpha_2(\tfrac{t}{n})\inv)}&=\norm{\overline{\alpha}_3(\tfrac{t}{n})*(-\overline{\alpha}_2(\tfrac{t}{n}))}\\
&\leq \norm{\overline{\alpha}_3(\tfrac{t}{n})-\overline{\alpha}_2(\tfrac{t}{n})}+C\frac{\norm{\overline{\alpha}_3(\tfrac{t}{n})}^2+\norm{\overline{\alpha}_2(\tfrac{t}{n})}^2}{2}\\
&\leq \frac{R\epsilon_1'}{n}+\frac{CR^2(\norm{\alpha_2'(0)}^2+\epsilon_1')}{n^2}\leq \frac{2R\epsilon_1'}{n}.
\end{aligned}
\end{equation*}

Let $\epsilon_1''>0$. Since the linear action $G\times\n\to\n,(g,v)\mapsto \dot{\pi}(g)v$ is continuous by Lemma~\ref{lemma:continuity_Tpi}, we may choose $\epsilon_1'$ small enough so that 
$$\norm{\dot{\pi}(g)v}\leq \epsilon_1''\quad\textrm{for all $g\in G$ with $\dist(g,1_G)\leq\epsilon_1'$ and all $v\in\n$ with $\norm{v}\leq 2R\epsilon_1'$.}$$
Moreover, since $\beta_1$ is continuous, there is some $N_1''\geq N_1'$ such that $\dist(\beta_1(\tfrac{t}{n}),1_G)\leq \epsilon_1'$ for all $t\in [0,R]$ and $n\geq N_1''$.
Hence for all $t\in [0,R]$ and $n\geq N_1''$,
\begin{equation*}
\begin{aligned}
\norm{n\cdot \exp_V\inv(\pi(\beta_1(\tfrac{t}{n}))(\alpha_3(\tfrac{t}{n})\alpha_2(\tfrac{t}{n})\inv))}=\norm{\dot{\pi}(\beta_1(\tfrac{t}{n}))(n\cdot\exp_V\inv(\alpha_3(\tfrac{t}{n})\alpha_2(\tfrac{t}{n})\inv))}\leq\epsilon_1''.
\end{aligned}
\end{equation*}

Finally, we choose $\epsilon_1''$ so that $\epsilon_1''\cdot e^{\epsilon_1''}\leq\epsilon_1$.
Since $\alpha_1$ is continuous, there is some $N_1\geq N_1''$ such that $\norm{\overline{\alpha}_1(t/n)}\leq \epsilon_1''$ for all $t\in [0,R]$ and $n\geq N_1$.
It then follows from Claim~1.4 that for all $t\in [0,R]$ and $n\geq N_1$,
\begin{equation*}
\begin{aligned}
\norm{n\overline{x}(\tfrac{t}{n})}&=\norm{n\cdot \exp_V\inv\big(\alpha_1(\tfrac{t}{n})\cdot\pi(\beta_1(\tfrac{t}{n}))(\alpha_3(\tfrac{t}{n})\alpha_2(\tfrac{t}{n})\inv)\cdot \alpha_1(\tfrac{t}{n})\inv\big)}\\
&=\norm{e^{\ad \overline{\alpha}_1(t/n)}(n\cdot \exp_V\inv(\pi(\beta_1(\tfrac{t}{n}))(\alpha_3(\tfrac{t}{n})\alpha_2(\tfrac{t}{n})\inv)))}\leq\epsilon_1''\cdot e^{\epsilon_1''}\leq\epsilon_1.
\end{aligned}
\end{equation*}
This concludes the proof of Claim~2.1.

\smallskip
\noindent
$\bullet$ \underline{Claim 2.2}: \emph{Let $\epsilon_2>0$. Then there is some $N_2\in\NN$ such that for all $n\geq N_2$, $s\in\{0,\dots,n\}$ and $t\in [0,R]$, $$z_s(\tfrac{t}{n})\inv\cdot x(\tfrac{t}{n})\cdot z_s(\tfrac{t}{n})\in U\quad\textrm{and}\quad\norm{\exp_V\inv\big(z_s(\tfrac{t}{n})\inv\cdot x(\tfrac{t}{n})\cdot z_s(\tfrac{t}{n})\big)}\leq \epsilon_2/n.$$}
\\
Indeed, let $\epsilon_2'>0$. Define the continuous curve $z\co \RR\to N$ by $$(z(t),\exp_G(t(\beta_1'(0)+\beta_2'(0))))=\exp_H(tc_2'(0))\quad\textrm{for all $t\in I$}.$$
Then (\ref{eqn:c2_Trotter}) implies that $z_n(\tfrac{t}{n})\stackrel{n\to\infty}{\longrightarrow}z(t)$ uniformly for $t\in [0,R]$. 
Hence there exists some $N_2'\in\NN$ (depending on $\epsilon_2'$) such that $$z_n(\tfrac{t}{n})z(t)\inv\in U\quad\textrm{and}\quad \norm{\exp_V\inv(z_n(\tfrac{t}{n})z(t)\inv)}\leq\epsilon_2'\quad\textrm{for all $t\in [0,R]$ and $n\geq N_2'$.}$$
In particular, 
\begin{equation}\label{eqn:zzs}
\norm{\exp_V\inv(z_s(\tfrac{t}{n})z(t)\inv)}=\norm{\exp_V\inv(z_s(\tfrac{st/n}{s})z(t)\inv)}\leq\epsilon_2'\quad\textrm{for all $t\in [0,R]$ and $n\geq s\geq N_2'$.}
\end{equation}
Note also that by Claim~2.1, there is some $N_1\geq N_2'$ such that 
\begin{equation}\label{eqn:claim21}
x(\tfrac{t}{n})\in U\quad\textrm{and}\quad\norm{\overline{x}(\tfrac{t}{n})}\leq \epsilon_2'/n\quad\textrm{for all $t\in [0,R]$ and $n\geq N_1$.}
\end{equation}

Consider the compact subset $K:=z([0,R])\inv$ of $N$, and let $U_1\subseteq U$ be an identity neighbourhood in $N$ such that $\mathrm{Int}(K)U_1:=\{gug\inv \ | \ g\in K, \ u\in U_1\}\subseteq U$.
From (\ref{eqn:zzs}), (\ref{eqn:claim21}) and Claim~1.4, we get for all $t\in [0,R]$, $n\geq N_1$ and $s\leq n$ with $s\geq N_2'$ that
\begin{equation}\label{eqn:aaa}
\mathrm{Int}(z(t))(z_s(\tfrac{t}{n})\inv\cdot x(\tfrac{t}{n})\cdot z_s(\tfrac{t}{n}))=\exp_N(e^{\ad\exp_V\inv(z(t)z_s(\tfrac{t}{n})\inv)}\overline{x}(\tfrac{t}{n}))
\end{equation}
and that 
\begin{equation}\label{eqn:aab}
\norm{e^{\ad\exp_V\inv(z(t)z_s(\tfrac{t}{n})\inv)}\overline{x}(\tfrac{t}{n})}\leq \tfrac{\epsilon_2'}{n}\cdot e^{\epsilon_2'}.
\end{equation}
In particular, choosing $\epsilon_2'$ small enough, we may assume that 
$$z_s(\tfrac{t}{n})\inv\cdot x(\tfrac{t}{n})\cdot z_s(\tfrac{t}{n})\in \mathrm{Int}(z(t)\inv)U_1\subseteq \mathrm{Int}(K)U_1\subseteq U$$
for all $t\in [0,R]$, $n\geq N_1$ and $s\leq n$ with $s\geq N_2'$. Let now $V_2\subseteq V$ be a $0$-neighbourhood in $\n$ such that $\mathrm{Ad}(K)V_2\subseteq \overline{B}^{\n}_{\epsilon_2}(0)$. We further choose $\epsilon_2'$ small enough so that $\overline{B}^{\n}_{\epsilon_2'\cdot e^{\epsilon_2'}}(0)\subseteq V_2$. (We now fix this choice of $\epsilon_2'$, and hence also of $N_2'$.)
We then deduce from (\ref{eqn:aaa}) and (\ref{eqn:aab}) that for all $t\in [0,R]$, $n\geq N_1$ and $s\leq n$ with $s\geq N_2'$,
\begin{equation*}
n\cdot \mathrm{Ad}(z(t))(\exp_V\inv(z_s(\tfrac{t}{n})\inv\cdot x(\tfrac{t}{n})\cdot z_s(\tfrac{t}{n})))= n\cdot\exp_V\inv\big(\mathrm{Int}(z(t))(z_s(\tfrac{t}{n})\inv\cdot x(\tfrac{t}{n})\cdot z_s(\tfrac{t}{n}))\big)\in V_2
\end{equation*}
and hence that
\begin{equation*}
n\cdot \exp_V\inv(z_s(\tfrac{t}{n})\inv\cdot x(\tfrac{t}{n})\cdot z_s(\tfrac{t}{n}))\in \mathrm{Ad}(K)V_2\subseteq \overline{B}^{\n}_{\epsilon_2}(0).
\end{equation*}
This proves Claim~2.2 when $s\geq N_2'$, with $N_2:=N_1$. 

We now deal with the case where $s< N_2'$. In this case, since $z_s$ is continuous, there is some $N_2\geq N_1$ such that $$z_s(\tfrac{t}{n})\in U\quad\textrm{and}\quad \norm{\exp_V\inv(z_s(\tfrac{t}{n}))}\leq\epsilon'_2\quad\textrm{for all $t\in [0,R]$ and $n\geq N_2$.}$$
Up to increasing $N_2$, we may further assume that $z_s(\tfrac{t}{n})\inv\cdot x(\tfrac{t}{n})\cdot z_s(\tfrac{t}{n})\in U$ for all $t\in [0,R]$ and $n\geq N_2$.
Together with (\ref{eqn:claim21}) and Claim~1.4, this implies that
$$z_s(\tfrac{t}{n})\inv\cdot x(\tfrac{t}{n})\cdot z_s(\tfrac{t}{n})=\exp_N\big(e^{\ad \exp_V\inv(z_s(\tfrac{t}{n})\inv)}\overline{x}(\tfrac{t}{n})\big)$$
and 
$$\norm{e^{\ad \exp_V\inv(z_s(\tfrac{t}{n})\inv)}\overline{x}(\tfrac{t}{n})}\leq \tfrac{\epsilon_2'}{n}\cdot e^{\epsilon_2'}$$
for all $t\in [0,R]$ and $n\geq N_2$. Hence, as before,
$$z_s(\tfrac{t}{n})\inv\cdot x(\tfrac{t}{n})\cdot z_s(\tfrac{t}{n})\in U$$
and $$\norm{n\cdot\exp_V\inv\big(z_s(\tfrac{t}{n})\inv\cdot x(\tfrac{t}{n})\cdot z_s(\tfrac{t}{n})\big)}=\norm{n\cdot e^{\ad \exp_V\inv(z_s(\tfrac{t}{n})\inv)}\overline{x}(\tfrac{t}{n})}\leq \epsilon_2'\cdot e^{\epsilon_2'}\leq \epsilon_2$$
for all $t\in [0,R]$ and $n\geq N_2$. This concludes the proof of Claim~2.2.

\smallskip
\noindent
$\bullet$ \underline{Claim 2.3}: \emph{Let $\epsilon_3>0$. Then there is some $N_3\in\NN$ such that $$F_s(n,t)\in U\quad\textrm{and}\quad\norm{\exp_V\inv(F_s(n,t))}\leq \epsilon_3/n\quad\textrm{for all $t\in [0,R]$, $n\geq N_3$ and $s\in\{0,\dots,n\}$.}$$}
\\
Indeed, let $\epsilon_3'>0$. 
Define the continuous curve $$\widetilde{y}\co \RR\to G, \ \widetilde{y}(t):=\exp_G(-t(\beta_1'(0)+\beta_2'(0))).$$
The strong Trotter property of the Banach Lie group $G$ then implies that $y(\tfrac{t}{n})^{-n}\stackrel{n\to\infty}{\longrightarrow}\widetilde{y}(t)$ uniformly for $t\in [0,R]$. 
Hence there exists some $N_3'\in\NN$ such that $$\dist(y(\tfrac{t}{n})^{-n},\widetilde{y}(t))\leq\epsilon_3'\quad\textrm{for all $t\in [0,R]$ and $n\geq N_3'$.}$$
In particular, 
\begin{equation}\label{eqn:2.3.1}
\dist(y(\tfrac{t}{n})^{-s},\widetilde{y}(t))=\dist(y(\tfrac{st/n}{s})^{-s},\widetilde{y}(t))\leq\epsilon_3'\quad\textrm{for all $t\in [0,R]$ and $n\geq s\geq N_3'$.}
\end{equation}
Note also that by Claim~2.2, there is some $N_2\geq N_3'$ such that for all $n\geq N_2$, $s\in\{0,\dots,n\}$ and $t\in [0,R]$, 
\begin{equation}\label{eqn:2.3.2}
z_s(\tfrac{t}{n})\inv\cdot x(\tfrac{t}{n})\cdot z_s(\tfrac{t}{n})\in U\quad\textrm{and}\quad\norm{\exp_V\inv\big(z_s(\tfrac{t}{n})\inv\cdot x(\tfrac{t}{n})\cdot z_s(\tfrac{t}{n})\big)}\leq \epsilon_3'/n.
\end{equation}

Consider the compact subset $K:=\widetilde{y}([0,R])$ of $G$. Since the linear action $G\times\n\to\n,(g,v)\mapsto \dot{\pi}(g)v$ is continuous by Lemma~\ref{lemma:continuity_Tpi}, one may choose $\epsilon_3'$ small enough so that 
\begin{equation}\label{eqn:2.3.3}
\dot{\pi}(K\overline{B}^{G}_{\epsilon_3'}(1_G))\overline{B}^{\n}_{\epsilon_3'}(0)\subseteq \overline{B}^{\n}_{\epsilon_3}(0)\cap V,
\end{equation}
where the closed ball in $G$ is defined with respect to the metric $\dist$. We now fix this choice of $\epsilon_3'$, and hence also of $N_3'$.

From (\ref{eqn:2.3.1}), (\ref{eqn:2.3.2}) and (\ref{eqn:2.3.3}), we get for all $t\in [0,R]$, $n\geq N_2$ and $s\leq n$ with $s\geq N_3'$ that
\begin{equation*}
\dot{\pi}(y(\tfrac{t}{n})^{-s})\exp_V\inv(z_s(\tfrac{t}{n})\inv\cdot x(\tfrac{t}{n})\cdot z_s(\tfrac{t}{n}))\in V,
\end{equation*}
so that
\begin{equation*}
F_s(n,t)=\pi(y(\tfrac{t}{n})^{-s})(z_s(\tfrac{t}{n})\inv\cdot x(\tfrac{t}{n})\cdot z_s(\tfrac{t}{n}))=\exp_N\big(\dot{\pi}(y(\tfrac{t}{n})^{-s})\exp_V\inv(z_s(\tfrac{t}{n})\inv\cdot x(\tfrac{t}{n})\cdot z_s(\tfrac{t}{n}))\big)\in U,
\end{equation*}
and that 
\begin{equation*}
\norm{n\cdot \exp_V\inv(F_s(n,t))}=\norm{\dot{\pi}(y(\tfrac{t}{n})^{-s})\big(n\cdot \exp_V\inv(z_s(\tfrac{t}{n})\inv\cdot x(\tfrac{t}{n})\cdot z_s(\tfrac{t}{n}))\big)}\leq \epsilon_3.
\end{equation*}
This proves Claim~2.3 when $s\geq N_3'$, with $N_3:=N_2$.

We now deal with the case where $s< N_3'$. In this case, since $y^{-s}$ is continuous, there is some $N_3\geq N_2$ such that $$y(\tfrac{t}{n})^{-s}\in \overline{B}^{G}_{\epsilon_3'}(0)\quad\textrm{for all $t\in [0,R]$ and $n\geq N_3$.}$$
One then concludes as above that for all $t\in [0,R]$ and $n\geq N_3$,
$$F_s(n,t)\in U\quad\textrm{and}\quad\norm{n\cdot\exp_V\inv(F_s(n,t))}\leq \epsilon_3.$$
This concludes the proof of Claim~2.3.

\smallskip
\noindent
$\bullet$ \underline{Claim 2.4}: \emph{Let $\epsilon>0$. Then there is some $N\in\NN$ such that $$F(n,t)\in U\quad\textrm{and}\quad\norm{\exp_V\inv(F(n,t))}\leq \epsilon \quad\textrm{for all $t\in [0,R]$ and $n\geq N$.}$$}
\\
Indeed, let $\rho>0$ be such that $\overline{B}^{\n}_{\rho}(0)\subseteq V_1$. Denote for each $n\in\NN$ by $\widetilde{n}$ the smallest natural number of the form $\widetilde{n}=2^k$ for some $k\in\NN$, such that $\widetilde{n}\geq n$. 
By Claim~2.3, there is some $N\in\NN$ such that $$F_s(n,t)\in U\quad\textrm{and}\quad\norm{\exp_V\inv(F_s(n,t))}\leq \tfrac{\rho}{2n}\leq\rho/\widetilde{n}\quad\textrm{for all $t\in [0,R]$, $n\geq N$ and $s\in\{0,\dots,n\}$.}$$
It then follows from Claims~1.2 and 1.3 that for all $t\in [0,R]$ and $n\geq N$,
$$F(n,t)=\prod_{s=1}^{n}{F_s(n,t)}=\exp_V(\exp_V\inv(F_1(n,t))*\dots*\exp_V\inv(F_n(n,t)))\in U$$
and
\begin{equation*}
\norm{\exp_V\inv(F(n,t))}=\norm{\exp_V\inv(F_1(n,t))*\dots*\exp_V\inv(F_n(n,t))}\leq a_{\widetilde{n}}(\rho/\widetilde{n})\leq a(\rho)=\frac{1}{C}(e^{C\rho}-1).
\end{equation*}
Choosing $\rho$ such that $a(\rho)\leq\epsilon$ then yields the claim.

It now follows from Claim~2.4 that $F(n,t)\stackrel{n\to\infty}{\longrightarrow} 1_N$ uniformly for $t\in [0,R]$. This concludes the proof of the theorem.
\end{proof}

\begin{corollary}\label{corollary:TP}
Let $G,N$ be Banach Lie groups and $\pi\co G\to \Aut(N)$ be a continuous action. Then $H=N\rtimes_{\pi}G$ has the Trotter property: for all $C^1$ (local) one-parameter subgroups $\gamma_1,\gamma_2\co I\to H$,
\begin{equation*}
\lim_{n\to\infty}{\big(\gamma_1(\tfrac{t}{n})\gamma_2(\tfrac{t}{n})\big)^n}=\exp_H\big(t(\gamma_1'(0)+\gamma_2'(0))\big)
\end{equation*}
uniformly in $t$ on compact subsets of $[0,\infty[$.
\end{corollary}
\begin{proof}
This readily follows from Lemma~\ref{lemma:property_star} and Theorem~\ref{thm:TFimproved}.
\end{proof}

\subsection{Continuous one-parameter subgroups of \texorpdfstring{$H$}{H}}\label{section:COPSOH}
We conclude this section by investigating the space $\Hom^0(\RR,H)$ of continuous one-parameter subgroups of $H=N\rtimes_{\pi}G$. In contrast to the situation of Banach--Lie groups, not every $\gamma\in\Hom^0(\RR,H)$ is automatically smooth: using the non-smoothness of the multiplication in $H$, one can easily produce counterexamples by conjugating smooth one-parameter subgroups with suitable elements. If $N$ is abelian, we prove that, conversely, every $\gamma\in\Hom^0(\RR,H)$ is conjugated to some smooth one-parameter subgroup of $H$.

Given a continuous representation $\mathcal U\co\RR\to\GL(E), t\mapsto \mathcal U_t$ of $\RR$ on a Banach space $E$, we call a map $\alpha\co\RR\to E$ a \emph{$1$-cocycle} (or just \emph{cocycle}) for $\mathcal U$ if $\mathcal U_t\alpha(s)=\alpha(t+s)-\alpha(t)$ for all $s,t\in\RR$. Cocycles of the form $\alpha(t)=\mathcal U_tv-v$ for some $v\in E$ 
are called \emph{coboundaries}. 
We call two cocycles \emph{equivalent} if their difference is a coboundary. 
\begin{lemma}\label{lemma:continuous_cocycle}
Let $\mathcal U\co\RR\to\GL(E),t\mapsto \mathcal U_t$ be a continuous action of $\RR$ on a Banach space $E$. Then every continuous cocycle $\alpha\co\RR\to E$ for $\mathcal U$ is equivalent to a smooth cocycle for $\mathcal U$.
\end{lemma}
\begin{proof}
Pick some bump function $f\in C^{\infty}_c(\RR)$ with $\int_{\RR}f(s)ds=1$, and set $v:=\int_{\RR}f(s)\alpha(s)ds\in E$. Consider the coboundary 
$\widetilde{\alpha}(t):=\mathcal U_tv-v$. Then for all $t\in\RR$, we have
\begin{equation*}
\begin{aligned}
\widetilde{\alpha}(t)&=\int_{\RR}f(s)(\mathcal U_t\alpha(s)-\alpha(s))ds=\int_{\RR}f(s)(\alpha(t+s)-\alpha(t)-\alpha(s))ds\\
&=\int_{\RR}f(s)(\alpha(t+s)-\alpha(s))ds-\alpha(t)=\int_{\RR}(f(s-t)-f(s))\alpha(s)ds-\alpha(t).
\end{aligned}
\end{equation*}
Hence $\widetilde{\alpha}(t)+\alpha(t)=\int_{\RR}(f(s-t)-f(s))\alpha(s)ds$ is a smooth cocycle for $\mathcal U$, yielding the lemma.
\end{proof}

\begin{prop}\label{prop:cont_1param}
If $N$ is abelian, then every continuous one-parameter subgroup $\gamma\co \RR\to H=N\rtimes_{\pi}G$ is conjugated to a smooth one, hence of the form 
$\gamma(t) = g\exp_H(tx) g^{-1}$ for some $g \in H, x\in \hh$.
\end{prop}

\begin{proof}
Assume first that $N=E$ is a Banach space. Write $\gamma(t)=(\alpha(t),\beta(t))$ for some continuous curves $\alpha\co\RR\to N$ and $\beta\co\RR\to G$. Since $\gamma$ is a one-parameter subgroup, $\beta$ is a (smooth) one-parameter subgroup of $G$ and $$\alpha(s+t)=\pi(\beta(t))\alpha(s)+\alpha(t)\quad\textrm{for all $s,t\in\RR$}.$$ In other words, $\alpha$ is a cocycle for the continuous $\RR$-action $\mathcal U:=\pi\circ\beta\co\RR\to\GL(N)$. It then follows from Lemma~\ref{lemma:continuous_cocycle} that there is some $v\in N$ such that $\widetilde{\alpha}(t):=\alpha(t)+\pi(\beta(t))v-v$ is a smooth cocycle for $\pi\circ\beta$. Hence $(\widetilde{\alpha}(t),\beta(t))=(v,1_G)\inv\gamma(t)(v,1_G)$ is a smooth one-parameter subgroup of $H$, as desired.

Assume next that $N$ is abelian. Since $\gamma(\RR)$ is connected, we may without loss of generality assume that $N$ is connected. Then the simply connected cover $\widetilde{N}$ of $N$ is a Banach space, and the action $\pi\co G\to \Aut(N)$ lifts to a continuous action $\widetilde{\pi}\co G\to \GL(\widetilde{N})$. One may then conclude using the first part of the proof.  
\end{proof}

\section{The Fr\'echet--Lie group \texorpdfstring{$N^{\infty}$}{} of smooth elements}
We use, throughout this section, the notation of \S\ref{section:notation}. In particular, $G$ and $N$ are Banach--Lie groups with respective Lie algebras $\g$ and $\n$, and $\pi^{\wedge}\co G\times N\to N,(g,n)\mapsto\pi(g)n$ is a continuous action. We establish in this section that the space $$N^{\infty}=\{n\in N \ | \ \textrm{$\pi^n\co G\to N, \ g\mapsto \pi(g)n$ is smooth}\}$$ of smooth elements for $\pi$ admits a Fr\'echet--Lie group structure for which the restricted action $$\pi^{\wedge}_{\infty}\co G\times N^{\infty}\to N^{\infty}, \ (g,n)\mapsto \pi^{\wedge}(g,n)$$ is smooth. We moreover show that the Fr\'echet--Lie groups $N^{\infty}$ and $H^{\infty}:=N^{\infty}\rtimes_{\pi}G$ are $R$-regular, and hence have the strong Trotter and commutator properties.

\subsection{A Fr\'echet--Lie group structure on \texorpdfstring{$H^{\infty}$}{the group of smooth elements}}\label{section:FLGSOH}
By \cite[Theorem~4.4 and Prop.~5.4]{Ne10b}, the space $\n^{\infty}=\{v\in\n \ | \ \textrm{$\dot{\pi}^v$ is smooth}\}$ of smooth vectors for $\dot{\pi}$ admits a Fr\'echet space topology for which the restricted action $\dot{\pi}^{\wedge}_{\infty}\co G\times \n^{\infty}\to \n^{\infty},(g,v)\mapsto \dot{\pi}(g)v$ is smooth. We will show that there exists a unique Fr\'echet--Lie group structure on $N^{\infty}$ for which the exponential map $\n^{\infty}\to N^{\infty},v\mapsto\exp_N(v)$ is a local diffeomorphism.

We recall from \cite[Definition~4.1]{Ne10b} that $\n^{\infty}$ is topologised as follows. Recall from \S\ref{section:notation} (see also \cite[Section~4]{Ne10b}) that we have a representation $d\dot{\pi}\co \g\to\End(\n^{\infty})$ of $\g$, defined by
$$d\dot{\pi}(x)v:=d(\dot{\pi}^v)(1_G;x)=\frac{d}{dt}\Big|_{t=0}\dot{\pi}(\exp_G(tx))v\quad\textrm{for all $x\in \g$ and $v\in\n^{\infty}$.}$$
In particular, in the notation of \S\ref{section:LG_NeJap} (see (\ref{eqn:k-th_diff})), we have for all $k\in\NN$, $x_1,\dots,x_k\in\g$ and $v\in \n^{\infty}$ that
\begin{equation}\label{eqn:dpi_d1}
\begin{aligned}
d\dot{\pi}(x_1)\dots d\dot{\pi}(x_k)v&=\frac{d}{dt_1}\Big|_{t_1=0}\dots\frac{d}{dt_k}\Big|_{t_k=0}\dot{\pi}(\exp_G(t_1x_1)\dots \exp_G(t_kx_k))v\\
&= d^k(\dot{\pi}^v)(1_G;x_1,\dots,x_k).
\end{aligned}
\end{equation}
For each $k\in\NN$, let $\Mult^k(\g,\n)$ be the space of continuous $k$-linear maps $\g^k\to\n$, which we equip with the topology of uniform convergence on bounded sets. Thus $\Mult^k(\g,\n)$ is a Banach space with respect to the norm $$\norm{\omega}:=\sup\{\norm{\omega(x_1,\dots,x_k)} \ : \ \norm{x_1},\dots,\norm{x_k}\leq 1\}.$$
Consider also the map
$$\Psi_k\co \n^{\infty}\to \Mult^k(\g,\n), \quad \Psi_k(v)(x_1,\dots,x_k):=d\dot{\pi}(x_1)\dots d\dot{\pi}(x_k)v.$$
This yields an injective linear map
$$\Psi\co\n^{\infty}\to \prod_{k\in\NN}{\Mult^k(\g,\n)},$$
and we define the topology on $\n^{\infty}$ so that $\Psi$ is a topological embedding. 
Then $\n^{\infty}$ is a Fr\'echet space, with respect to the family $\{p_k \ | \ k\in\NN\}$ of seminorms defined by $$p_k(v):=\sup\{\norm{d\dot{\pi}(x_1)\dots d\dot{\pi}(x_k)v} \ : \ \norm{x_1},\dots,\norm{x_k}\leq 1\}.$$

The following two lemmas provide our main tool to establish smoothness of $\n^{\infty}$-valued maps.

\begin{lemma}\label{lemma:technical_le_Gl}
Let $X,Y,Z$ be locally convex spaces and $U_Y\subseteq Y$ be open. Let $f\co U_Y\to \Mult^k(X,Z)$ be a map such that $\widetilde{f}\co U_Y\times X^k\to Z, (y,x_1,\dots,x_k)\mapsto f(y)(x_1,\dots,x_k)$ is smooth. Then $f$ is smooth.
\end{lemma}
\begin{proof}
This follows from \cite[Proposition~2.1(b)]{Gl07} or \cite[Corollary~1.6.32]{GlNe16}.
\end{proof}

\begin{lemma}\label{lemma:corestriction_smoothbis}
Let $F$ be one of the Fr\'echet spaces $\n^{\infty}$, $\LLL^{\infty}(I,\n^{\infty})$ or $C(I,\n^{\infty})$, and write $F_B$ for the Banach space $\n$, $\LLL^{\infty}(I,\n)$ or $C(I,\n)$, respectively. Let $E$ be a locally convex space, $U\subseteq E$ be an open subset and $h\co U\to F$ be a map. Assume that there exists some identity neighbourhood $U_G$ in $G$ such that the map $$f\co U_G\times U\to F_B, \ (g,y)\mapsto \dot{\pi}^{h(y)}(g)=\dot{\pi}(g)h(y)$$ is smooth. Then $h$ is smooth. 
\end{lemma}
\begin{proof}
Since $f$ is smooth, the maps 
$$\widetilde{f}_k\co \g^k\times U\to F_B, \ (x_1,\dots,x_k,y)\mapsto (d_1^kf)(1_G,x_1,\dots,x_k,y)=d^k(\dot{\pi}^{h(y)})(1_G,x_1,\dots,x_k)$$
are also smooth for each $k\in\NN$. Lemma~\ref{lemma:technical_le_Gl} then implies that for each $k\in\NN$, the induced map $$f_k\co U\to \Mult^k(\g,F_B), \ y\mapsto \big[(x_1,\dots,x_k)\mapsto d^k(\dot{\pi}^{h(y)})(1_G,x_1,\dots,x_k)\big]$$
is smooth. Since $f_k=\Psi_k\circ h$ by (\ref{eqn:dpi_d1}) (where $\Psi_k(\alpha)(t):=\Psi_k(\alpha(t))$ for $t\in I$ and $\alpha\in F$ if $F$ is $\LLL^{\infty}(I,\n^{\infty})$ or $C(I,\n^{\infty})$), this in turn implies that $h$ is smooth, as desired.
\end{proof}

We now introduce some additional notation, allowing us to work in charts.
\begin{lemma}\label{lemma:nNinfty}
There exists an open connected symmetric $0$-neighbourhood $V$ in $\n$ such that the following assertions hold:
\begin{enumerate}
\item There exists some open $0$-neighbourhood $W\subseteq \n$ with $V\subseteq W$ such that $\exp_W:=\exp_N|_W$ is a diffeomorphism onto $\exp_N(W)$, and such that $\exp_N(V)^2\subseteq \exp_N(W)$. In particular, one may define the local multiplication $*\co V\times V\to \n,(x,y)\mapsto x*y:=\exp_W\inv(\exp_W(x)\exp_W(y))$.
\item $\exp_N(V^{\infty})=N^{\infty}\cap \exp_N(V)$, where $V^{\infty}:=V\cap \n^{\infty}$.
\end{enumerate}
\end{lemma}
\begin{proof}
By Lemma~\ref{lemma:naturality_exp}, there exist open (connected, symmetric) neighbourhoods $U_G$ of $1_G$ in $G$ and $V$ of $0$ in $\n$ such that $\exp_V:=\exp_N|_V$ is a diffeomorphism onto $\exp_V(V)$ and such that  $$(\exp_{V})\inv\circ \pi^{\exp_N(v)}(g)=\dot{\pi}^v(g) \quad\textrm{for all $g\in U_G$ and $v\in V$.}$$ Up to schrinking $V$, we may moreover assume that (1) is satisfied.

If $v\in \n^{\infty}$, so that $\dot{\pi}^v$ is smooth, then $\pi^{\exp_N(v)}=\exp_N\circ \dot{\pi}^v$ is smooth (because $\exp_N$ is smooth), showing that $\exp_N(v)\in N^{\infty}$. Conversely, if $\exp_N(v)\in N^{\infty}$ for some $v\in V$, so that $\pi^{\exp_N(v)}$ is smooth, then $\dot{\pi}^v|_{U_G}=(\exp_{V})\inv\circ \pi^{\exp_N(v)}|_{U_G}$ is smooth, and hence also $\dot{\pi}^v$ (because $\dot{\pi}^v\circ\lambda_g=\dot{\pi}(g)\circ \dot{\pi}^v$ for all $g\in G$). Thus (2) is also satisfied.
\end{proof}

Let $V$ and $V^{\infty}$ be as in the statement of Lemma~\ref{lemma:nNinfty}. We also fix some open identity neighbourhood $U_G\subseteq G$ and some open (connected, symmetric) $0$-neighbourhood $V_1\subseteq V$ such that $\dot{\pi}(U_G)V_1\subseteq V$ (see Lemma~\ref{lemma:continuity_Tpi}). Up to schrinking $U_G$, we moreover assume that there is some open $0$-neighbourhood $V_G\subseteq \g$ such that $\exp_{V_G}:=\exp_G|_{V_G}$ is a diffeomorphism onto $U_G=\exp_G(V_G)$. Finally, we set $V_1^{\infty}:=V_1\cap V^{\infty}$. Note that, since the topology on $\n^{\infty}$ is finer than the topology on $\n$, the inclusion map $\iota\co \n^{\infty}\hookrightarrow \n$ is smooth. In particular, the sets $V^{\infty}$ and $V_1^{\infty}$ are open in $\n^{\infty}$. 

\begin{lemma}\label{lemma:nlocalLiegr}
The quadruple $(\n^{\infty},V_1^{\infty}\times V_1^{\infty},*,0)$ is a local Lie group, where
 $*\co V_1^{\infty}\times V_1^{\infty}\to \n^{\infty}$ is the local multiplication of $\n^{\infty}$.
\end{lemma}
\begin{proof}
Since the local inversion $\n^{\infty}\to \n^{\infty},x\to -x$ is clearly smooth, we only have to show that the local multiplication $*$ is smooth. By Lemma~\ref{lemma:corestriction_smoothbis}, it is sufficient to show that the map
$$f\co U_G\times (V_1^{\infty}\times V_1^{\infty})\to\n, \ (g,x,y)\mapsto \dot{\pi}^{x*y}(g)=\dot{\pi}(g)x*\dot{\pi}(g)y$$
is smooth. But $f$ is the composition of the smooth action map $$U_G\times (V_1^{\infty}\times V_1^{\infty})\to V^{\infty}\times V^{\infty}, \ (g,x,y)\mapsto (\dot{\pi}^{\wedge}(g,x),\dot{\pi}^{\wedge}(g,y)),$$ the smooth inclusion map $V^{\infty}\times V^{\infty}\hookrightarrow V\times V$ and the local multiplication map $V\times V\to\n,(x,y)\mapsto x*y$ in $\n$, yielding the claim.
\end{proof}

\begin{lemma}\label{lemma:Ad_smooth}
For each $n\in N^{\infty}$, the map $\Ad(n)\co \n^{\infty}\to\n^{\infty}$ is smooth.
\end{lemma}
\begin{proof}
Let $n\in N^{\infty}$. For all $g\in G$ and $v\in\n^{\infty}$, we have
\begin{equation*}
\begin{aligned}
\dot{\pi}(g)(\Ad(n)v)&=\frac{d}{dt}\Big|_{t=0}\pi(g)(n\exp_N(tv)n\inv)
=\frac{d}{dt}\Big|_{t=0}\Int(\pi(g)n)\pi(g)(\exp_N(tv))\\
&=\Ad(\pi(g)n)\dot{\pi}(g)v.
\end{aligned}
\end{equation*}
In particular, $\Ad(n)\n^{\infty}\subseteq\n^{\infty}$. Moreover, the map $$f\co U_G\times \n^{\infty}\to\n, \ (g,v)\mapsto \dot{\pi}^{\Ad(n)v}(g)=\Ad(\pi(g)n)\dot{\pi}(g)v$$
is smooth, as it is the composition of the smooth maps $$U_G\times \n^{\infty}\to N\times \n, \ (g,v)\mapsto (\pi^n(g),\iota\circ\dot{\pi}^{\wedge}(g,v))$$ and $\Ad\co N\times\n\to\n$. 
We may then apply Lemma~\ref{lemma:corestriction_smoothbis} to conclude that the map $\Ad(n)\co \n^{\infty}\to\n^{\infty}$ is smooth as well, as desired.
\end{proof}

\begin{theorem}\label{thm:Ninfty_Frechet}
The group $N^{\infty}$ has a unique Fr\'echet--Lie group structure for which $\exp_N\co \n^{\infty}\to N^{\infty}$ is a local diffeomorphism. 
\end{theorem}
\begin{proof}
By Lemma~\ref{lemma:nlocalLiegr}, the quadruple $(V^{\infty},V_1^{\infty}\times V_1^{\infty},*,0)$ is a local Fr\'echet--Lie group. 
Moreover, by Lemma~\ref{lemma:nNinfty}, the restriction of $\exp_N$ to $V^{\infty}$ yields an injective morphism $\exp_{V^{\infty}}\co V^{\infty}\to N^{\infty}$ of local groups. Set $U:=\exp_{V^{\infty}}(V^{\infty})\subseteq N^{\infty}$. Note that $U=U\inv$ as $V^{\infty}$ is symmetric. We equip $U$ with the smooth Fr\'echet manifold structure coming from $V^{\infty}$, that is, such that $\exp_{V^{\infty}}\co V^{\infty}\to U$ is a diffeomorphism. This turns $U$ into a local Lie group, with respect to the multiplication and neutral element in $N^{\infty}$.

For each $n\in N^{\infty}$, the conjugation map
$$\Int(n)\co \exp_N(V_1^{\infty})\to N^{\infty}, \ \exp_N(v)\mapsto n\exp_N(v)n\inv=\exp_N(\Ad(n)v)$$
is smooth by Lemma~\ref{lemma:Ad_smooth}, and hence there exists some open symmetric identity neighbourhood $U_n\subseteq U$ such that $\Int(n)U_n\subseteq U$ and such that $\Int(n)\co U_n\to U$ is smooth.
It then follows from \cite[Theorem~II.2.1]{NeJap} that there is a unique Lie group structure on $N^{\infty}$ for which $\exp_N\co \n^{\infty}\to N^{\infty}$ is a local diffeomorphism.
\end{proof}

We now show that the induced $G$-action $$\pi_{\infty}\co G\to \Aut(N^{\infty}), \ g\mapsto \pi(g)|_{N^{\infty}}$$ on the Fr\'echet--Lie group $N^{\infty}$ is smooth.

\begin{theorem}
The action map $\pi^{\wedge}_{\infty}\co G\times N^{\infty}\to N^{\infty},(g,n)\mapsto \pi(g)n$ is smooth.
\end{theorem}
\begin{proof}
We first claim that for any $n\in N^{\infty}$, the orbit map $\pi_{\infty}^n\co G\to N^{\infty},g\mapsto \pi(g)n$ is smooth. Indeed, let $n\in N^{\infty}$ and consider the smooth map $h_n:=\lambda_{n\inv}\circ\pi^n\co G\to N$. Since $h_n(1_G)=1_N$, there exists some open $0$-neighbourhood $W_G\subseteq V_G$ in $\g$ such that $$\pi(\exp_G(W_G))h_n(\exp_G(W_G))\subseteq \exp_N(V)\cap N^{\infty}=\exp_N(V^{\infty})$$ (see Lemma~\ref{lemma:nNinfty}). Since $\pi_{\infty}^n\circ\lambda_g=\pi(g)\circ\pi_{\infty}^n$ for all $g\in G$, it is sufficient to prove that $\pi_{\infty}^n$ is smooth on $\exp_G(W_G)$, or equivalently, that the map $$\widetilde{h}_n\co W_G\to \n^{\infty}, \ x\mapsto \exp_V\inv(h_n(\exp_G(x)))$$ is smooth. But since $h_n$ is smooth, the map $f\co \exp_G(W_G)\times W_G\to \n$ defined by
$$f(g,x):=\dot{\pi}^{\widetilde{h}_n(x)}(g)=\dot{\pi}(g)\exp_V\inv(n\inv\cdot\pi(\exp_G(x))n)=\exp_V\inv(h_n(g)\inv\cdot h_n(g\exp_G(x)))$$
is smooth as well, so that the claim follows from Lemma~\ref{lemma:corestriction_smoothbis}.

On the other hand, the restriction of $\pi^{\wedge}_{\infty}$ to the open subset $G\times \exp_N(V^{\infty})$ of $G\times N^{\infty}$ is smooth, since it is the composition of the smooth maps $$\id_{G}\times \exp_V\inv\co G\times \exp_N(V^{\infty})\to G\times V^{\infty},\quad\dot{\pi}^{\wedge}_{\infty}\co G\times \n^{\infty}\to \n^{\infty}\quad\textrm{and}\quad\exp_N\co\n^{\infty}\to N^{\infty}.$$ Given $n_0\in N^{\infty}$, it now remains to prove the smoothness of the map $$G\times n_0\exp_N(V^{\infty})\to N^{\infty}, \ (g,n_0n)\mapsto \pi^{\wedge}_{\infty}(g,n_0n)=\pi(g)n_0\cdot\pi(g)n.$$ But this follows from the smoothness of the maps $G\times n_0\exp_N(V^{\infty})\to G\times \exp_N(V^{\infty}),(g,n_0n)\mapsto (g,n)$, $$G\times \exp_N(V^{\infty})\to N^{\infty}\times N^{\infty}, \ (g,n)\mapsto (\pi_{\infty}^{n_0}(g),\pi^{\wedge}_{\infty}(g,n))$$ and $N^{\infty}\times N^{\infty}\to N^{\infty},(n_1,n_2)\mapsto n_1n_2$.
\end{proof}

\begin{corollary}
The group $H^{\infty}:=N^{\infty}\rtimes_{\pi_{\infty}}G$ has a canonical Fr\'echet--Lie group structure extending the Lie group structures on $G$ and $N^{\infty}$.
\end{corollary}

\subsection{\texorpdfstring{$R$-regularity of $H^{\infty}$}{R-regularity of the group of smooth elements}}\label{section:RROH}
Keeping with the notation introduced so far in this section, we now show that the Fr\'echet--Lie group $N^{\infty}$ is $R$-regular, with a smooth evolution map $$\Evol_{N^{\infty}}\co R(I,\n^{\infty})\to AC_R(I,N^{\infty}).$$ Since for any $\alpha\in R(I,\n^{\infty})$, the map $\iota\circ\alpha$ is in $R(I,\n)$ by Lemma~\ref{lemma:RtoR} (where $\iota\co\n^{\infty}\hookrightarrow\n$), one may identify $R(I,\n^{\infty})$ with a subspace of $R(I,\n)$. It is then natural to consider the restriction of $\Evol_N\co R(I,\n)\to AC_R(I,N)$ to $R(I,\n^{\infty})$ as candidate for $\Evol_{N^{\infty}}$. 

Let $W\subseteq V_1$ be an open $0$-neighbourhood in $\n$ such that $$\Evol_N(R(I,W))\subseteq AC_R(I,\exp_N(V_1)).$$ Let $\widetilde{U}_G\subseteq U_G$ and $W_1\subseteq W$ be identity neighbourhoods such that $$\dot{\pi}(\widetilde{U}_G)W_1\subseteq W$$
and set $$W^{\infty}:=W\cap\n^{\infty}\subseteq V_1^{\infty}\quad\textrm{and}\quad W_1^{\infty}:=W_1\cap\n^{\infty}\subseteq W_{\infty}.$$

Recall that $\Evol_N\co R(I,\n)\to AC_R(I,N)$ is smooth (see Lemma~\ref{lemma:Banach_R_regular}). 
Consider the smooth map
$$\eta\co \Rcal(I,W)\to AC_R(I,V_1), \ \alpha\mapsto \exp_V\inv\circ \Evol_N(\alpha).$$ Note that the corresponding map $$\Rcal(I,W)\to C(I,V_1), \ \alpha\mapsto \eta(\alpha)$$ is also smooth (cf. \S\ref{section:ACM}). Moreover, Lemma~\ref{lemma:naturality_Evol} yields that
\begin{equation}\label{eqn:etapi_pieta}
\dot{\pi}(g)\circ \eta(\alpha)=\eta(\dot{\pi}(g)\circ\alpha)\quad\textrm{for all $g\in \widetilde{U}_G$ and $\alpha\in \Rcal(I,W_1)$.}
\end{equation}

\begin{lemma}\label{lemma:RIsmoothness}
The map $U_G\times R(I,\n^{\infty})\to R(I,\n^{\infty}), (g,\alpha)\mapsto \dot{\pi}(g)\circ\alpha$ is smooth.
\end{lemma}
\begin{proof}
We deduce from Lemma~\ref{lemma:technical_lemma2} applied to the smooth action map $$f\co V_G\times \n^{\infty}\to \n^{\infty}, \ (x,v)\mapsto \dot{\pi}(\exp_G(x))v$$ that the map $$\widetilde{f}\co V_G\times R(I,\n^{\infty})\to R(I,\n^{\infty}), \ (x,\alpha)\mapsto \dot{\pi}(\exp_G(x))\circ\alpha$$ is smooth. Since $V_G\to U_G,x\mapsto \exp_G(x)$ is a diffeomorphism, the conclusion follows.
\end{proof}

\begin{lemma}\label{lemma:Evolution_Ninfty}
The following assertions hold:
\begin{enumerate}
\item
For all $\alpha\in \Rcal(I,W_1^{\infty})$, the map $\widetilde{U}_G\to C(I,\n),g\mapsto \dot{\pi}(g)\circ \eta(\alpha)$ is smooth.
\item
$\eta(\alpha)(t)\in \n^{\infty}$ for all $t\in I$ and $\alpha\in \Rcal(I,W_1^{\infty})$.
\item
For all $\alpha\in \Rcal(I,W_1^{\infty})$, the map $\eta_{\alpha}\co I\to\n^{\infty}, t\mapsto \eta(\alpha)(t)$ is in $\mathcal L^{\infty}(I,\n^{\infty})$.
\item
The map $\Rcal(I,W_1^{\infty})\to \mathcal{L}^{\infty}(I,\n^{\infty}),\alpha\mapsto \eta(\alpha)$ is smooth.
\item
For all $\alpha\in \Rcal(I,W_1^{\infty})$, the map $\eta_{\alpha}$ is continuous.
\end{enumerate}
\end{lemma}
\begin{proof}
(1): Let $\alpha\in \Rcal(I,W_1^{\infty})$. Since the inclusion map $\Rcal(I,W^{\infty})\hookrightarrow \Rcal(I,W)$ is smooth by Lemma~\ref{lemma:RtoR}, the map $\widetilde{U}_G\to \Rcal(I,W),g\mapsto \dot{\pi}(g)\circ\alpha$ is also smooth by Lemma~\ref{lemma:RIsmoothness}. The claim then follows from (\ref{eqn:etapi_pieta}) and the smoothness of the map $\Rcal(I,W)\to C(I,\n),\alpha\mapsto \eta(\alpha)$. 

(2): Let $t\in I$ and $\alpha\in \Rcal(I,W_1^{\infty})$. We have to show that the map $\widetilde{U}_G\to \n,g\mapsto \dot{\pi}(g)\eta(\alpha)(t)$ is smooth. But this follows from (1) and the smoothness of the evaluation map $C(I,\n)\to \n,\beta\mapsto \beta(t)$.

(3): Let $\alpha\in \Rcal(I,W_1^{\infty})$. By definition, $\eta_{\alpha}\in\mathcal L^{\infty}(I,\n^{\infty})$ if $\sup_{t\in I}{p_k(\eta_{\alpha}(t))}<\infty$ for all $k\in\NN$. Fix $k\in\NN$. By (1), the map $\widetilde{U}_G\to C(I,\n),g\mapsto \dot{\pi}^{\eta_{\alpha}(\cdot)}(g)$ is smooth. Thus the multilinear map 
$$m_k\co\g^k\to C(I,\n), \ (x_1,\dots,x_k)\mapsto d^k(\dot{\pi}^{\eta_{\alpha}(\cdot)})(1_G;x_1,\dots,x_k)$$ is continuous, hence bounded. It then follows from (\ref{eqn:dpi_d1}) that
\begin{equation*}
\begin{aligned}
\sup_{t\in I}{p_k(\eta_{\alpha}(t))}&=\sup\{\norm{d^k(\dot{\pi}^{\eta_{\alpha}(t)})(1_G;x_1,\dots,x_k)} \ : \ t\in I, \ \norm{x_1},\dots,\norm{x_k}\leq 1\}\\
&=\sup\{\norm{m_k(x_1,\dots,x_k)}_{C(I,\n)} \ : \ \norm{x_1},\dots,\norm{x_k}\leq 1\}<\infty,
\end{aligned}
\end{equation*}
as desired.

(4): The map $$f\co \widetilde{U}_G\times \Rcal(I,W_1^{\infty})\to \mathcal{L}^{\infty}(I,\n), \ (g,\alpha)\mapsto \dot{\pi}^{\eta_{\alpha}(\cdot)}(g)=\eta(\dot{\pi}(g)\circ\alpha)$$ is smooth since it is the composition of the smooth maps $\widetilde{U}_G\times \Rcal(I,W_1^{\infty})\to \Rcal(I,W_{\infty}),(g,\alpha)\mapsto \dot{\pi}(g)\circ\alpha$ (see Lemma~\ref{lemma:RIsmoothness}) and $\Rcal(I,W_{\infty})\hookrightarrow \Rcal(I,W)\to C(I,\n)\hookrightarrow \mathcal{L}^{\infty}(I,\n),\alpha\mapsto \eta(\alpha)$. Hence (4) follows from Lemma~\ref{lemma:corestriction_smoothbis}.

(5): Let $\alpha\in \Rcal(I,W_1^{\infty})$, and let us show that $\eta_{\alpha}$ is continuous. This is clear if $\alpha$ is a constant function, that is, if $\alpha=c_v\co I\to W_1^{\infty},t\mapsto v$ for some $v\in W_1^{\infty}$, since then $\Evol_N(\alpha)(t)=\exp_N(tv)$, and hence $\eta_{\alpha}(t)=tv$ for all $t\in I$. Assume next that $\alpha$ is a step function, say $\alpha|_{]t_{j-1},t_j[}=c_{v_j}|_{]t_{j-1},t_j[}$ for $j=1,\dots,n$, where $0=t_0<t_1<\dots <t_n=1$ is a subdivision of $I$ and $v_1,\dots,v_n\in W_1^{\infty}$. Then for all $j\in\{1,\dots,n\}$ and $t\in [t_{j-1},t_j]$, we have
$$\eta_{\alpha}(t)=t_1v_1* (t_2-t_1)v_2*\dots * (t_{j-1}-t_{j-2})v_{j-1}*(t-t_{j-1})v_j,$$
so that the continuity of $\eta_{\alpha}$ follows from the continuity of the local multiplication in $\n^{\infty}$ (Lemma~\ref{lemma:nlocalLiegr}).

Let now $(\alpha_n)_{n\in\NN}$ be a sequence of step functions in $\Rcal(I,W_1^{\infty})$ converging to $\alpha$. Let $p$ be a continuous seminorm on $\n^{\infty}$. 
Fix $\varepsilon>0$. By (4), there exists some $N\in\NN$ such that $$\sup_{s\in I}{p(\eta_{\alpha}(s)-\eta_{\alpha_N}(s))}\leq\varepsilon/3.$$ The above discussion also yields some $\delta>0$ such that $p(\eta_{\alpha_N}(t)-\eta_{\alpha_N}(t_0))\leq\varepsilon/3$ whenever $|t-t_0|<\delta$. Thus for $t\in I$ with $|t-t_0|<\delta$, we have
\begin{equation*}
\begin{aligned}
p(\eta_{\alpha}(t)-\eta_{\alpha}(t_0))\leq p(\eta_{\alpha}(t)-\eta_{\alpha_N}(t))+p(\eta_{\alpha_N}(t)-\eta_{\alpha_N}(t_0))+p(\eta_{\alpha_N}(t_0)-\eta_{\alpha}(t_0))\leq\varepsilon.
\end{aligned}
\end{equation*}
Hence $\eta_{\alpha}$ is continuous, concluding the proof of the lemma. 
\end{proof}

\begin{lemma}\label{lemma:localevol_Ninfty}
For all $\alpha\in \Rcal(I,W_1^{\infty})$, the map $\eta_{\alpha}\co I\to\n^{\infty},t\mapsto \eta(\alpha)(t)$ is in $AC_R(I,\n^{\infty})$. Moreover, $R(I,W_1^{\infty})\to AC_R(I,N^{\infty}),\alpha\mapsto \Evol_N(\alpha)$ is a (local) evolution map for $N^{\infty}$.
\end{lemma}
\begin{proof}
Recall from \S\ref{section:LG_NeJap} that the map $$f\co V_1^{\infty}\times\n^{\infty}\to\n^{\infty}, \ (x,v)\mapsto \exp_N(x).v=T_1(\lambda_{\exp_N(x)})v$$ is continuous (here we identify each fiber $T_nN^{\infty}$ of $TN^{\infty}$ for $n\in \exp_N(V_1^{\infty})$ with the space $\n^{\infty}$, using the chart $(\exp_N(V_1^{\infty}),\exp_V\inv)$). Since $\eta_{\alpha}\co I\to \n^{\infty}$ is continous by Lemma~\ref{lemma:Evolution_Ninfty}(5) and $\alpha\in \Rcal(I,\n^{\infty})$, we deduce from Lemma~\ref{lemma:RtoR} that $$\gamma:=f\circ(\eta_{\alpha},\alpha)\co I\to\n^{\infty}, \ s\mapsto \Evol_N(\alpha)(s).\alpha(s)$$ is in $\Rcal(I,\n^{\infty})$. Since $\n^{\infty}$ is a Fr\'echet space, Lemma~\ref{lemma:thm_fond_CDI} then implies that the weak integrals $\widetilde{\eta}_{\alpha}(t):=\int_0^t{\gamma(s)ds}$ exist in $\n^{\infty}$ for all $t\in I$ and that $\widetilde{\eta}_{\alpha}\in AC_R(I,\n^{\infty})$.
On the other hand, by definition of $\Evol_N$ (cf. \S\ref{section:RReg}), we have $$\eta_{\alpha}(t)=\int_0^t{\gamma(s)ds}$$ for all $t\in I$, where the above weak integrals are considered in $\n$ (c.f. \S\ref{section:integration}). By unicity of the weak integral in $\n$, we conclude that $\eta_{\alpha}=\widetilde{\eta}_{\alpha}\in AC_R(I,\n^{\infty})$, as desired.

Finally, since $\exp_N\circ \eta(\alpha)=\Evol_N(\alpha)$ for all $\alpha\in \Rcal(I,W_1^{\infty})$, the above discussion also implies that $\Evol_N|_{R(I,W_1^{\infty})}\co R(I,W_1^{\infty})\to AC_R(I,N^{\infty})$ is a (local) evolution map for $N^{\infty}$.
\end{proof}

\begin{theorem}\label{thm:Rregular_Ninfty}
The Fr\'echet--Lie group $N^{\infty}$ is $R$-regular, with smooth evolution map $$\Evol_{N^{\infty}}\co R(I,\n^{\infty})\to AC_R(I,N^{\infty}), \ \alpha\mapsto \Evol_N(\alpha).$$
\end{theorem}
\begin{proof}
By Lemma~\ref{lemma:localevol_Ninfty}, the Fr\'echet--Lie group $N^{\infty}$ is locally $R$-semiregular (with local evolution map $\Evol_{N}|_{R(I,W_1^{\infty})}$). It then follows from Lemma~\ref{lemma:5.25}(1) that $N^{\infty}$ is $R$-semiregular. Moreover, the unicity of the evolution map for $N$ implies that $\Evol_{N^{\infty}}:=\Evol_{N}|_{R(I,\n^{\infty})}$ is the (global) evolution map for $N^{\infty}$. It thus remains to prove that $\Evol_{N^{\infty}}$ is smooth. By Lemma~\ref{lemma:5.25}(2), it is sufficient to show that $\Evol_{N^{\infty}}|_{R(I,W_1^{\infty})}$ is smooth or, equivalently, that the map $R(I,W_1^{\infty})\to AC_R(I,\n^{\infty}),\alpha\mapsto \eta(\alpha)$ is smooth. By Lemma~\ref{lemma:5.20}, this is in turn equivalent to the smoothness of the map $R(I,W_1^{\infty})\to C(I,\n^{\infty}),\alpha\mapsto \eta(\alpha)$. But this can be established exactly as in the proof of Lemma~\ref{lemma:Evolution_Ninfty}(4), replacing $\mathcal{L}^{\infty}(I,\n^{\infty})$ by $C(I,\n^{\infty})$ and $\mathcal{L}^{\infty}(I,\n)$ by $C(I,\n)$.
\end{proof}

\begin{corollary}
The Fr\'echet--Lie group $H^{\infty}=N^{\infty}\rtimes_{\pi^{\infty}}G$ is $R$-regular, with a smooth evolution map. In particular, $H^{\infty}$ has the strong Trotter and commutator properties.
\end{corollary}
\begin{proof}
The $R$-regularity of $H^{\infty}$ follows from Theorem~\ref{thm:Rregular_Ninfty}, together with the fact that $R$-regularity is an extension property (see \cite[Theorem~G]{Gl15}). The second statement then follows from \cite[Theorem~I]{Gl15}.
\end{proof}

\def\cprime{$'$}
\providecommand{\bysame}{\leavevmode\hbox to3em{\hrulefill}\thinspace}
\providecommand{\MR}{\relax\ifhmode\unskip\space\fi MR }
\providecommand{\MRhref}[2]{%
  \href{http://www.ams.org/mathscinet-getitem?mr=#1}{#2}
}
\providecommand{\href}[2]{#2}

\end{document}